\documentclass{amsart}
\usepackage{amsmath,amssymb}
\usepackage{graphicx}
\usepackage{tikz}
\usetikzlibrary{arrows.meta,decorations.markings}
\usepackage{microtype}
\usepackage{xcolor}
\usepackage[hidelinks]{hyperref}
\DeclareTextCommand{\k}{OT1}[1]{\UseTextAccent{T1}{\k}{#1}}

\allowdisplaybreaks[4]

\newtheorem{thm}{Theorem}
\newtheorem{lem}[thm]{Lemma}

\newtheorem{addendum}[thm]{Addendum}
\theoremstyle{definition}
\newtheorem{rmk}[thm]{Remark}

\newcommand{\Z}{\mathbb{Z}}
\newcommand{\CP}{{\mathbb{CP}}{}^{2}}
\newcommand{\CPb}{\overline{\mathbb{CP}}{}^{2}}
\newcommand{\x}{\times}
\newcommand{\F}{\mathbb{F}}
\newcommand{\id}{\mathrm{id}}
\newcommand{\Int}{\mathrm{int}}
\newcommand{\im}{\mathrm{im}}
\newcommand{\Mod}{\operatorname{Mod}}

\begin{document}

\title[Exotic $S^2\x S^2$ candidates that are just fake]
{Exotic $S^2\x S^2$ candidates that are just fake}

\author[R. \.{I}. Baykur]{R. \.{I}nan\c{c} Baykur}
\address{Department of Mathematics and Statistics, University of Massachusetts, Amherst, MA 01003, USA}
\email{inanc.baykur@umass.edu}

\date{}

\begin{abstract}
We show that two infinite families of homology $S^2\x S^2$'s presented in \cite{FPS, AP}, originally proposed as exotic candidates, are aspherical.
In particular, none of these manifolds is an exotic $S^2\x S^2$.
We prove this by reconstructing them from aspherical building blocks,
following Edmonds. The same arguments apply to and rule out further
candidates for simply connected exotica.
We also give counterexamples to the conjecture that Luttinger surgery
preserves symplectic asphericity.
\end{abstract}

\maketitle
\setcounter{secnumdepth}{2}

\section{Introduction}

The search for small exotic $4$--manifolds has led to many
constructions by surgeries along Lagrangian tori in symplectic
$4$--manifolds with fairly simple topology. For some of the smallest
proposed simply connected examples of signature zero, determining
the fundamental group has remained the main difficulty. In this
note, we consider two infinite families of homology $S^2\x S^2$'s
constructed in~\cite{FPS,AP} by Fintushel--Park--Stern and
Akhmedov--Park, respectively. We show that these manifolds are all
aspherical, and in particular, none is an exotic
$S^2\x S^2$.

This work was prompted by Paul Kirk's recent paper~\cite{Kirk},
whose calculations gave us a complete and convincing list of relators together with the peripheral data used in our argument for the torus surgeries of higher multiplicity. Our
arguments for asphericity are inspired by Allan Edmonds' beautiful
article~\cite{Edmonds}, whose ideas we had long wanted to put to use.
Much like in the constructions of exotic candidates we consider,
Edmonds works with products of punctured surfaces and assembles
aspherical $4$--manifolds by gluing along $\pi_1$-injective boundary
submanifolds.

There is also a more practical motivation for this article. The existence
of an exotic $S^2\x S^2$ remains open~\cite[Problem~4.3(b)]{K3}.
Claims and attempts to make the manifolds we consider here, or variations of them, simply
connected, with the hope of obtaining small exotica, have persisted
for many years; see also Remark~\ref{rmk:further}. While nobody wants to be
the bearer of bad news, we feel that explaining why these
constructions cannot work will be helpful, so nobody finds themselves still beating the same drum. We hope that
understanding what not to try will help guide the search for
constructions that do work.

We use \emph{fake $Z$} for a smooth closed
$4$--manifold with the same integral homology and intersection form as
$Z$ but not diffeomorphic to it, and call it an \emph{exotic $Z$} if it is
also homeomorphic to $Z$. We thus show that:

\begin{thm}\label{thm:families}
Every member of the two infinite families of fake
$S^2\x S^2$'s constructed in~\cite{FPS,AP} is aspherical; in particular,
none is homeomorphic to $S^2\x S^2$.
\end{thm}

Their fundamental groups are therefore infinite, have no torsion,
and are perfect---the last part being the luring reason for many to work on these examples. In fact, the boundary injections in our proof give another useful way to
show that these fundamental groups are nontrivial, without using
asphericity. The same injections also rule out metrics of strictly
negative sectional curvature; see Addendum~\ref{add:curvature}.

I should note that Ron Fintushel and Ron Stern conjectured that these
examples were symplectically aspherical, and suggested
(in private communication) approaching the question through metric
properties of the initial manifold. The broader conjecture that Luttinger surgery preserves
symplectic asphericity appears in~\cite{Fintushel}.
Indeed, symplectic asphericity alone would already rule out finite
fundamental groups. Recall that a symplectic manifold $(X,\omega)$
is \emph{symplectically aspherical} if the cohomology class
$[\omega]$ pairs trivially with every spherical homology class. If such
a closed symplectic $4$--manifold had finite fundamental group, its
universal cover $p\colon \widetilde X\to X$ would be closed and simply
connected. Every class in $H_2(\widetilde X;\Z)$ would then be spherical
by the Hurewicz theorem, so $[p^*\omega]=0$. This contradicts
$[p^*\omega]^2\ne0$, since $p^*\omega$ is a symplectic form on the
closed manifold $\widetilde X$. Our proof of Theorem~\ref{thm:families} takes a topological path instead and establishes
the stronger conclusion of ordinary asphericity for the families of constructions considered here. 

In fact, we also observe that symplectic asphericity is \emph{not} preserved under Luttinger surgery in general:

\begin{thm}\label{thm:nonpreservation}
There are many closed symplectically aspherical $4$--manifolds such that a Luttinger surgery gives a symplectic $4$--manifold admitting no symplectically aspherical form.
\end{thm}

Let me note that  such examples can also be spotted in constructions of exotic $4$--manifolds in the existing literature; see Remark~\ref{rmk:nonpreservation}. We add this result to emphasize that the failure of Luttinger surgeries to construct simply connected exotica here is really about the chosen models, and not the operation itself.

Here is the outline of the paper: We review Edmonds' approach and establish the gluing and torus surgery
lemmas in Section~\ref{sec:gluing}, including a comparison with Kirk's
presentation for the model with two surgery tori.
The two families of fake
$S^2\x S^2$'s are treated in Sections~\ref{sec:fps}
and~\ref{sec:ap}, respectively. For each family, we first prove
asphericity for the symplectic member, then extend the argument to the remaining members obtained by torus surgeries of higher multiplicity.
The proof for the symplectic examples (and thus the symplectic asphericity) is shorter, since their
product quadrants are indeed surface bundles. The more technical Lemma~\ref{lem:multiplicity}
supplies the additional local analysis for higher multiplicity torus surgeries, whereas the boundary gluing argument in Lemma~\ref{lem:boundary} is needed
for handling models that are not products as in \cite{AP}.
The AP example illustrates how our arguments can be applied to more general surface bundle models instead of surface products. In Section~\ref{sec:nonpreservation}, we prove
Theorem~\ref{thm:nonpreservation}. Further
general conclusions are discussed in Section~\ref{sec:addenda}.

\subsection*{Conventions}
We use the following conventions for surfaces, groups and torus surgeries.
All manifolds and maps below are smooth.
We compose paths from left to right.
We write $\Sigma_g^b$ for a compact connected oriented surface of
 genus $g$ with $b$ boundary components, and $\Sigma_g=\Sigma_g^0$.
We use
$[a,b]=aba^{-1}b^{-1}$, and $\F(a,b,\ldots)$ denotes the free group
on the displayed generators. We write $\Z_n$ for the cyclic group
of order $n\ge1$, and use the convention $\Z_0=\Z$.
For a surface $S$, $\Mod(S)$ denotes its orientation preserving
mapping class group; maps and isotopies fix $\partial S$ pointwise.
A framed torus surgery is specified by the image of the
meridian of the attached $T^2\x D^2$; a boundary diffeomorphism
preserving this meridian extends over $T^2\x D^2$. We write $\mu$
for the original meridian and $\lambda$ for a chosen framed
longitude. Slopes $\mu+k\lambda$ have multiplicity $1$, whereas
$m\mu\pm\lambda$ have multiplicity $m$. We keep this distinction
throughout.

\subsection*{Acknowledgements}
The author was partially supported by the NSF grant 
DMS-2506431 and by a Simons Foundation Travel Grant. We thank Ron
Fintushel for bringing the asphericity question back to our attention
while making suggestions for the $\mathrm{K3}$ problem list~\cite{K3}.
We thank Allan Edmonds, Ron Fintushel, Paul Kirk, Anubhav Mukherjee, Lisa Piccirillo, Ron Stern, Nathan Sunukjian, Zolt\'{a}n Szab\'{o} and Cliff Taubes for their interest in this work and their kind comments and questions. I'm particularly grateful to Paul for his thorough and careful work
on the subject, which has been inspirational, and for his encouragement
and help in brainstorming a fun title for this article.
OpenAI's ChatGPT was used for figure generation, and final editing.
The author takes full responsibility for all the mathematical
arguments and results.

\section{Aspherical gluings and torus surgery}\label{sec:gluing}

We review Edmonds' construction method~\cite{Edmonds}, which suggests
gluing products of punctured surfaces along their $\pi_1$-injective
faces to produce aspherical $4$--manifolds. We first discuss the torus
surgery models and boundary injections needed for these gluings.
The following form of Whitehead's theorem~\cite{Whitehead} is the key
input.

\begin{thm}[Whitehead]\label{thm:whitehead}
Let a manifold be obtained by gluing finitely many aspherical pieces along
disjoint aspherical boundary submanifolds. If the inclusions of
each gluing submanifold into the adjoining pieces are
$\pi_1$-injective, then the resulting manifold is aspherical.
The pieces and gluing submanifolds remain $\pi_1$-injective.
\end{thm}

The standard proof constructs the universal cover from contractible
lifts of the pieces, glued along contractible lifts of the gluing
submanifolds. These lifts are arranged along a tree, so the resulting
space is contractible; see~\cite[Theorem~1B.11]{Hatcher}.
To identify two disjoint boundary faces of a single piece, we can insert a
cylinder between them and apply the same theorem, as explained
in~\cite{Edmonds}. The boundary faces used for a gluing can
be $\pi_1$-injective even when the full boundary is not---which gives more flexibility to the construction.

\enlargethispage{3pt}
We use the topological form of the local model in~\cite{ADK} to
realize torus surgery of multiplicity $1$ as such a regluing.
Suppose a framed torus $T$ lies in the interior of a
hypersurface $Y$ with trivial normal bundle in a $4$--manifold. We require its prescribed framing
to agree with the one formed by a normal vector to $T$ within $Y$
and a normal vector to $Y$. In a collar
$T^2\x[-1,1]\subset Y$, use coordinates
$(\theta_1,\theta_2,v)$, with the circle coordinates taken modulo one.
For integers $r,s$, the map
\begin{equation}\label{eq:torussurgery}
(\theta_1,\theta_2,v)\longmapsto
(\theta_1+r\rho(v),\theta_2+s\rho(v),v),
\end{equation}
where $\rho$ is zero near $-1$ and one near $1$, extends by the
identity to a diffeomorphism of $Y$. Cutting along $Y$ and changing
its gluing by this map is the torus surgery with meridian slope
$\mu+r\lambda_1+s\lambda_2$, up to the consistent choice of meridian
orientation. To see this, let $w$ be the normal coordinate to $Y$,
so that the normal square is cut along $w=0$. On the half $w\ge0$,
change the torus coordinates to
$(\theta_1-r\rho(v),\theta_2-s\rho(v))$. The identification across
$w=0$ becomes the identity, so the two halves form $T^2\x D^2$
after rounding corners. In the resulting boundary parametrization,
the map~\eqref{eq:torussurgery} acts nontrivially only over one edge of
the normal square, while the other three edges are unchanged. Tracing this
boundary gives the claimed meridian slope.

Thus torus surgery leaves the pieces obtained by cutting along $Y$
unchanged and composes their boundary identification with the
diffeomorphism~\eqref{eq:torussurgery}.\footnote{%
This is the same idea used by Lickorish~\cite{Lickorish} in proving
that every closed connected oriented $3$--manifold is obtained by
surgery on a link in $S^3$: he changes the gluing of two handlebodies
by Dehn twists and realizes these twists by surgeries.}
This map is the identity near
$\partial Y$. For several disjoint hypersurfaces, we can make these
changes independently.

We apply this to a product of tori $P,Q$, each with one boundary component.
Consider two product tori
$\gamma_j\x\eta_j\subset\Int(P\x Q)$, where $\eta_1,\eta_2$
are disjoint parallel nonseparating curves in $Q$. The curves
$\gamma_1,\gamma_2$ in $P$ need not be disjoint (and often in these choices of surgeries, they won't be disjoint). The two boundary faces
$P\x\partial Q$ and $\partial P\x Q$ meet in the
\emph{corner torus} $\partial P\x\partial Q$. We call the image
of its fundamental group in the surgered quadrant the \emph{corner subgroup}.
As in~\cite{Edmonds}, we will glue along the
individual product faces. The following calculation identifies their
fundamental groups as subgroups of the fundamental group of the
surgered quadrant, and shows that their intersection is the corner subgroup.

\begin{lem}\label{lem:product}
Perform surgeries on the two tori with slopes $\mu_j+k_j\gamma_j$ in their product
framings. The resulting manifold $Z$ is a $P$-bundle over $Q$, agreeing
with the original product near both boundary faces. These faces inject
on fundamental groups, and their images intersect precisely in the
image of the corner torus. The manifold $Z$ is aspherical.
\end{lem}

\begin{proof}
We put the two $\eta_j$ at distinct levels of an annulus in $Q$.
In~\eqref{eq:torussurgery}, we use $\theta_1$ in the $\gamma_j$
direction and $\theta_2$ in the $\eta_j$ direction. The prescribed
slope then gives $(r,s)=(k_j,0)$, so the regluing is
$t_{\gamma_j}^{k_j}$ on the fiber and the identity in the base
direction. Here $t_{\gamma_j}$ is the Dehn twist with the sign
determined by our surgery convention. Performing the two regluings
in their order along the annulus gives a $P$-bundle over $Q$.
Their supports miss the product collars, and a fixed point
$p\in\partial P$ gives the section $s\colon Q\to Z$, $s(q)=(p,q)$.

We take the basepoint in the corner and write
\[
F=\pi_1(P,p),\qquad B=\pi_1(Q,q),\qquad
u=[\partial P],\quad\delta=[\partial Q].
\]
The homotopy exact sequence gives
$1\to F\to\pi_1(Z)\to B\to1$, which the section $s$ splits.
Thus $\pi_1(Z)=F\rtimes_\rho B$, where $\rho$ is the monodromy action
on $F$. Each element has a unique expression in this semidirect product, which we
write as $(f,b)$.

We choose generators $c,d$ of $B$, with $c$ parallel to the $\eta_j$
and $d$ crossing their annulus. Their monodromies are $\id$ and
$\varphi=t_{\gamma_1}^{k_1}t_{\gamma_2}^{k_2}$, respectively,
with the factors numbered in the order of composition.
With $\delta=[c,d]$, the boundary monodromy is
\[
[\id,\varphi]=\id\quad\text{in }\Mod(P).
\]
Since the fiber boundary is fixed pointwise,
\[
\rho(b)(u)=u\quad(b\in B),\qquad \rho(\delta)=\id.
\]
The images of $P\x\partial Q$ and $\partial P\x Q$ are
\begin{equation}\label{eq:faces}
K=\{(f,\delta^n)\mid f\in F,\ n\in\Z\},\qquad
L=\{(u^m,b)\mid b\in B,\ m\in\Z\}.
\end{equation}
The boundary elements $u=[\partial P]\in F$ and
$\delta=[\partial Q]\in B$ are nontrivial commutators in free groups
and hence have infinite order. The two inclusion maps are
$(f,n)\mapsto(f,\delta^n)$ from $F\x\Z$ and
$(m,b)\mapsto(u^m,b)$ from $\Z\x B$, which are clearly injective. A pair lies in both images exactly when its first
coordinate is a power of $u$ and its second is a power of $\delta$.
Therefore
\begin{equation}\label{eq:corner}
K\cap L=\langle u,\delta\rangle=\im \, \pi_1(\partial P\x\partial Q).
\end{equation}
Finally, it follows from the homotopy exact sequence of the surface bundle
that $Z$ is aspherical.
\end{proof}

In Lemma~\ref{lem:product}, we allow arbitrary simple closed curves on the fiber. 
For a pair intersecting once, Kirk's presentation of the complement
of the two surgery tori~\cite[Theorem~9.2]{Kirk} gives the same semidirect product
description after adding the surgery relations. Using the
framed longitudes in the $\gamma_j$-directions, we obtain
\begin{equation}\label{eq:kirk}
\left\langle x,y,t,z\ \middle|\
[t,x]=[t,y]=1,\quad zyz^{-1}=yx^p,\quad
zxz^{-1}=(yx^p)^{-r}x\right\rangle.
\end{equation}
Here $p,r\in\Z$ encode the two surgery coefficients. In Kirk's
coordinates the based longitude in the $\gamma_2$-direction on
the second surgery torus is $zyz^{-1}$, not $y$. The conjugation
comes from his choice of basing path.
Using this longitude gives~\eqref{eq:kirk}.
The base generator $z$ acts on the fiber group $\F(x,y)$ by
the automorphism induced by the surface bundle monodromy:
\[
y\longmapsto yx^p,\qquad x\longmapsto(yx^p)^{-r}x.
\]
This automorphism fixes the fiber boundary word $[y,x^{-1}]$.
Thus~\eqref{eq:kirk} is the semidirect product described in the proof
above, with the other base generator $t$ acting trivially.

In all the constructions we will consider, we pass from the
symplectic member to the rest of the family by changing the
multiplicity of the surgery on a single torus in a product quadrant. We can make this change independently in several distinct
product quadrants, provided the surgery on the other torus in each
quadrant still has slope $\mu\pm\gamma$. The lemma is applied to
each surgered quadrant before gluing.
\footnote{For the pair of torus surgeries in Lemma~\ref{lem:product}, one can view the resulting closed $4$--manifold as a $T^2$--bundle over $T^2$ as explained in 
\cite[Section~3.2.2]{BaykurThesis}.}
Here we need the boundary information for the punctured pieces, including
torus surgeries of higher multiplicity.
The following version of
Lemma~\ref{lem:product} shows that asphericity and the boundary injections
needed for gluing remain intact under these torus surgeries.

\begin{lem}\label{lem:multiplicity}
In the product model above, suppose $\gamma_1$ and $\gamma_2$ intersect
once transversally. Perform surgery on one torus with slope $\mu\pm\gamma$ and on the other with
slope $m\mu\pm\gamma$, where $m\ge1$. The resulting quadrant is
aspherical. Both product faces inject on fundamental groups, and the intersection of their images is exactly the corner subgroup.
\end{lem}

\begin{proof}
A diffeomorphism of $P$ interchanges the essential unoriented curves
$\gamma_1,\gamma_2$. A diffeomorphism
of $Q$ interchanges $\eta_1,\eta_2$, the two boundary circles of
the annulus they cobound. Their product therefore exchanges the
two surgery tori, with any changes of sign allowed in the statement.
Thus we may take the varying surgery to be the second one.

The torus $Q$ with one boundary component is obtained from an annulus by
attaching a band joining its two boundary circles. We choose this
annulus to contain $\eta_1,\eta_2$ as parallel core curves, as in
Figure~\ref{fig:annulus-band}. Removing
the band leaves $I\x S^1_t$, so the corresponding piece of
$P\x Q$ is $P \x (I \x S^1_t) = (P\x I)\x S^1_t$.
We first carry out both surgeries with this band in $Q$ removed,
and reattach it afterward.
Both torus surgeries are products with $S^1_t$ of Dehn surgeries
on curves in level surfaces of $P\x I$. For the surgery of
multiplicity $1$, the description used in Lemma~\ref{lem:product}
cuts $P\x I$ along a level surface and reglues by a Dehn twist
about $\gamma_1$. The result is again $P\x I$, with the
identification of one end changed by this twist. The remaining
surgery is the one with slope $m\mu\pm\gamma_2$. Under this product
identification, its surgery curve is still nonseparating in a level
surface.

Recall that $P \cong \Sigma_1^1$. We can therefore view
$P$ as an annulus with the nonseparating surgery curve as core,
together with a band joining its two boundary circles. A properly
embedded arc across this band is disjoint from the surgery curve and
cuts $P$ back to the annulus. (Similar to the picture for $Q$ given in  Figure~\ref{fig:annulus-band}.)
Taking products with $I$ turns the annulus into a solid torus and
the band into a $3$--dimensional $1$--handle. This describes $P\x I$
as a handlebody of genus $2$, with the surgery curve as a core of the
solid torus.

The complement of a tubular neighborhood of the core in that
solid torus is $S^1 \x S^1 \x I$.
Dehn surgery with any primitive slope
glues a solid torus to one end of this collar, giving a solid torus
again. The $1$--handle is untouched, so the surgered $3$--manifold is
again a handlebody of genus $2$ $H$. Taking its product with the unchanged
circle factor gives the surgered strip $H\x S^1_t$.

A handlebody retracts onto a graph, so $H\x S^1_t$ is
aspherical. We next check that the two end surfaces inject into $H$
on fundamental groups; this is what allows us to reattach the band
using Whitehead's theorem --- which we will do in a bit.

\begin{figure}[htbp]
\centering
\definecolor{curveblue}{RGB}{0,105,170}
\definecolor{pathorange}{RGB}{191,104,16}
\definecolor{pathpurple}{RGB}{117,77,143}
\definecolor{loopgreen}{RGB}{0,120,98}
\tikzset{
  surface/.style={line width=.55pt,draw=black!65},
  curve/.style={line width=.9pt,line cap=round},
  arrow/.style={postaction={decorate},decoration={markings,
    mark=at position .56 with {\arrow{Stealth[length=4pt,width=3.4pt]}}}},
  paired/.style={postaction={decorate},decoration={markings,
    mark=at position .5 with {\arrow{Stealth[length=4pt,width=3.4pt]}\arrow[xshift=-4pt]{Stealth[length=4pt,width=3.4pt]}}}}
}
\resizebox{.44\textwidth}{!}{%
\begin{tikzpicture}[x=1.45cm,y=1.05cm,font=\normalsize]
  \fill[black!3] (0,0) rectangle (4,2.8);
  \path[fill=pathpurple!7,draw=black!60,line width=.55pt]
    (0,1.10) .. controls (-1.12,1.10) and (-1.12,4.45) .. (0,4.45)
    -- (4,4.45)
    .. controls (5.12,4.45) and (5.12,1.10) .. (4,1.10)
    -- (4,1.70)
    .. controls (4.55,1.70) and (4.55,3.73) .. (4,3.73)
    -- (0,3.73)
    .. controls (-.55,3.73) and (-.55,1.70) .. (0,1.70)
    -- cycle;
  \draw[surface,paired] (0,0) -- (4,0);
  \draw[surface,paired] (0,2.8) -- (4,2.8);
  \draw[surface] (0,0) -- (0,1.10) (0,1.70) -- (0,2.8);
  \draw[surface] (4,0) -- (4,1.10) (4,1.70) -- (4,2.8);
  \draw[black!30,densely dashed,line width=.4pt] (0,1.10)--(0,1.70) (4,1.10)--(4,1.70);
  \node[black!65] at (2,3.38) {band};
  \draw[black!50,line width=.4pt,-{Stealth[length=3pt]}] (2,3.56)--(2,3.85);

  \draw[curve,curveblue,arrow] (1.28,0) -- (1.28,2.8);
  \draw[curve,curveblue,arrow] (2.72,0) -- (2.72,2.8);
  \node[curveblue,anchor=west] at (1.35,.54) {$\eta_1$};
  \node[curveblue,anchor=east] at (2.65,.54) {$\eta_2$};

  \draw[curve,loopgreen,arrow] (0,0) -- (0,2.8);
  \draw[curve,loopgreen,arrow] (4,0) -- (4,2.8);
  \node[loopgreen,anchor=east] at (-.12,.55) {$t$};
  \node[loopgreen,anchor=west] at (4.12,.55) {$t_+$};

  \draw[curve,pathorange,arrow] (0,1.40) -- (4,1.40);
  \node[pathorange,fill=black!3,inner sep=2pt] at (2,1.11) {$\alpha$};
  \draw[curve,pathpurple,arrow]
    (0,1.40) .. controls (-.835,1.40) and (-.835,4.09) .. (0,4.09)
    -- (4,4.09)
    .. controls (4.835,4.09) and (4.835,1.40) .. (4,1.40);
  \node[pathpurple,fill=pathpurple!7,inner sep=2pt] at (2,4.09) {$\beta$};
  \fill (0,1.40) circle[radius=1.45pt];
  \fill (4,1.40) circle[radius=1.45pt];

  \node[anchor=north,font=\small] at (2,-.24)
    {The top and bottom edges are identified.};
  \node[anchor=north] at (2,-.75)
    {$z=\alpha\beta^{-1},\qquad
      \delta=\beta t_+\beta^{-1}t^{-1}=z^{-1}tzt^{-1}.$};
  \coordinate (figSW) at (current bounding box.south west);
  \coordinate (figNE) at (current bounding box.north east);
  \path[use as bounding box]
    ([xshift=-5pt,yshift=-5pt]figSW) rectangle
    ([xshift=5pt,yshift=5pt]figNE);
\end{tikzpicture}%
}
\caption{The decomposition of $Q \cong \Sigma_1^1$ into an annulus and a band. The horizontal edges of the
rectangle are identified. Both $\alpha$ and $\beta$ run from the left
point to the right, and $t,t_+$ are oriented in the same direction.}
\label{fig:annulus-band}
\end{figure}
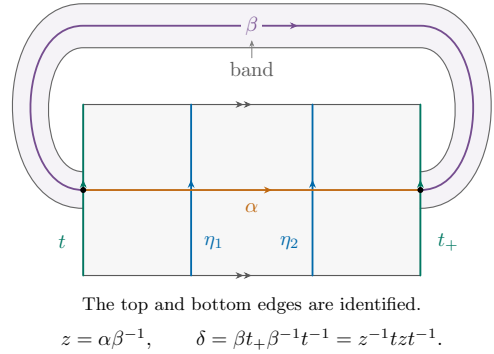

To compute the inclusions of the two end surfaces, we return to
$P\x I$ before either surgery. We remove open tubular
neighborhoods of its two surgery curves and call the resulting
complement $D$. Now, Kirk's calculation for the cut model~\cite[Proof of Proposition~2.1]{Kirk} gives
\[
\pi_1(D)=\langle x,a,q,h\mid[x,a]=[q,h]=1\rangle,
\]
where $(a,x)$ and $(h,q)$ are the meridian--longitude pairs on the boundary tori of the two removed solid torus neighborhoods,
with longitudes given by the product framings.\footnote{
The local group and peripheral data can also be verified
using Kirby diagrams for these building blocks and torus surgeries,
as in~\cite[Section~2]{Akbulut}; we feel Kirk's calculations
in~\cite{Kirk} give a more accessible way to do this for most readers.}
The two end surfaces are $P\x\{0\}$ and
$P\x\{1\}$. Write $i_-$ and $i_+$ for the maps on fundamental
groups induced by their inclusions into $D$, using the base paths
in that calculation. With the same basis $(x,y)$ for the fiber fundamental group at the two
ends, the geometric inclusions give
\[
i_-\colon (x,y)\longmapsto(x,qa),\qquad
i_+\colon (x,y)\longmapsto(hx,q).
\]
The common circle direction is $t$. The surgeries impose
$ax^p=1$ and $h^m q^r=1$.
As in~\eqref{eq:kirk}, $p,r$ denote the coefficients of the chosen longitudes;
here each is a sign $\pm1$. Put $N=rm$, so $N=\pm m\ne0$. Then
we have $a=x^{-p}$, $q=h^{-N}$ and $\pi_1(H)=\F(x,h)$.
The two end surfaces of the original $P\x I$ remain in $\partial H$. Their inclusion
maps on fundamental groups, still denoted by $i_-$ and $i_+$, are
\[
\begin{aligned}
i_-(x)&=x,& i_-(y)&=h^{-N}x^{-p},\\
i_+(x)&=hx,& i_+(y)&=h^{-N}.
\end{aligned}
\]
For the fundamental group of the surgered piece with the band in $Q$
still removed, write
\[
U=\pi_1(H\x S_t^1)=\F(x,h)\x\langle t\rangle.
\]
The end maps into this group have images
\[
A_-=\langle x,h^N\rangle,\qquad A_+=\langle hx,h^N\rangle.
\]
The pair $(hx,h)$ is another free basis of $\F(x,h)$. The pairs $(x,h^N)$ and $(hx,h^N)$ freely generate $A_-$ and
$A_+$, respectively: since $N\ne0$, a nonempty reduced word in
either pair remains nontrivial when expanded in the corresponding
ambient basis $(x,h)$ or $(hx,h)$.
For the first end, the free basis $(x,yx^p)$ maps to
$(x,h^{-N})$; for the second, $(x,y)$ maps to $(hx,h^{-N})$.
Hence both end maps are injective.

They agree on the fiber boundary word:
\[
u=i_-([y,x^{-1}])=i_+([y,x^{-1}])=h^{-N}x^{-1}h^Nx.
\]
Since $N\ne0$, this word is nontrivial in $\F(x,h)$, so $u$
has infinite order.
The two end boundary circles cobound the annulus
$\partial P\x I$, which is disjoint from both surgeries.
With compatible base paths, this annulus shows that their boundary
loops give the same element of $\pi_1(H)$.

We now restore the band. Move the fiber basepoint to $\partial P$
along a path disjoint from the drilled curves, conjugating all
inclusions simultaneously. Besides the annular loop $t$, the band
supplies a second base loop $z$. These generate the base group on
the face $\partial P\x Q$.
Keeping the fiber basepoint fixed, view the paths $\alpha,\beta$
in $Q$ shown in Figure~\ref{fig:annulus-band} as paths in the quadrant.
Both run from the $P\x\{0\}$ end to the $P\x\{1\}$ end,
through the annulus and the band, respectively. Put
$z=\alpha\beta^{-1}$. If $t_+$ is the positive annular loop at the
other end, the untouched annulus at the fiber basepoint gives
$\alpha t_+\alpha^{-1}=t$. Following the base boundary
gives the based boundary loop
$\delta=\beta t_+\beta^{-1}t^{-1}=z^{-1}tzt^{-1}$.
The band is a product, so its attachment identifies the end groups by
$zi_-(f)z^{-1}=i_+(f)$ for $f\in\pi_1(P)$. Since we have shown
that $i_-$ and $i_+$ are injective,
Theorem~\ref{thm:whitehead} makes the resulting quadrant aspherical.

We now check the two boundary faces and the intersection of their
images. By van Kampen, the fundamental group of the quadrant is
\begin{equation}\label{eq:multiplicityHNN}
V=\left\langle x,h,t,z\ \middle|\
[t,x]=[t,h]=1,\quad zxz^{-1}=hx,\quad
zh^Nz^{-1}=(hx)^{-p}h^N\right\rangle.
\end{equation}
Theorem~\ref{thm:whitehead} also gives $U\hookrightarrow V$.
The presentation~\eqref{eq:multiplicityHNN} defines a surjective homomorphism
\[
\rho\colon V\longrightarrow\F(t,z),\qquad
\rho(x)=\rho(h)=1,\quad \rho(t)=t,\quad\rho(z)=z.
\]
Every defining relation maps to the identity. A relation between
$t,z$ in $V$ would also hold in $\F(t,z)$. Thus we see that $t,z$
generate a free subgroup of $V$, which we identify with $\F(t,z)$.
The word $\delta=z^{-1}tzt^{-1}$ and all its nonzero powers are
reduced and nonempty, so $\delta$ has infinite order in this free
group and hence in $V$.

First consider $P\x\partial Q$, whose fundamental group has image
\[
K=\langle A_-,\delta\rangle.
\]
The element $\delta$ commutes with $A_-$. Indeed, both end
groups lie in the factor $\F(x,h)$ of $U$, so $t$ commutes with
them. For $a\in A_-$, we have $zaz^{-1}\in A_+$, hence
$\delta a\delta^{-1}=z^{-1}t(zaz^{-1})t^{-1}z=a$; that is, $\delta$ and $a$ commute in $V$. If $a\delta^n=1$ in $V$, with $a\in A_-$, applying $\rho$
gives $\delta^n=1$ in $\F(t,z)$, so $n=0$. Then $a=1$, since
$A_-\le U$ injects into $V$. Hence
$K\cong A_-\x\langle\delta\rangle$, so the inclusion of this
face induces an injection on fundamental groups.

Next consider $\partial P\x Q$, whose fundamental group has image
\[
L=\langle u,t,z\rangle.
\]
Since $i_-([y,x^{-1}])=i_+([y,x^{-1}])=u$, the relation
$zi_-(f)z^{-1}=i_+(f)$, applied to $f=[y,x^{-1}]$,
gives $zuz^{-1}=u$. The loop $t$ also commutes with $u$, since
$u\in\F(x,h)$ and $U=\F(x,h)\x\langle t\rangle$.
Every element of $L$ can therefore be written $u^jv$, with
$v\in\F(t,z)$. If $u^jv=1$, applying $\rho$ gives $v=1$.
Since $u$ has infinite order in $U\le V$, we then have $j=0$.
Thus $L\cong\langle u\rangle\x\F(t,z)$, proving the second
face injection.

The same homomorphism $\rho$ determines the intersection. Write an
element of $K\cap L$ as $a\delta^n=u^jv$, with $a\in A_-$ and
$v\in\F(t,z)$. Applying $\rho$ gives $v=\delta^n$ in $\F(t,z)$,
and hence also in its image in $V$. Cancelling this common factor
gives $a=u^j$. Conversely, both $u$ and $\delta$ belong to both
face images. Thus $K\cap L=\langle u,\delta\rangle$, as required.
\end{proof}

When we glue these pieces, we must also check that the remaining
boundary injects. The calculation of the corner intersection  above is what
makes this possible. First, let $G=G_1*_E G_2$ be an amalgamated
product, where the two maps from $E$
are injective. If $L_i\le G_i$ satisfy
$L_i\cap E=C$, then
\begin{equation}\label{eq:subamalgam}
L_1*_C L_2\hookrightarrow G_1*_E G_2.
\end{equation}
A reduced word on the left remains reduced on the right, since
an element of $L_i$ lies in $E$ exactly when it lies in $C$.
For two quadrants glued along one product face, $E$ is the
gluing face group, $L_i$ are the groups of the remaining faces, and
$C$ is the corner subgroup. Van Kampen identifies the remaining
boundary group with $L_1*_C L_2$, so~\eqref{eq:subamalgam} proves
its injection into the glued piece.

The AP construction we will look at in Section~\ref{sec:ap} requires
one further check for the injection of the remaining boundary.
There, a connected gluing face meets one boundary piece in two components.
A loop in the boundary can enter this piece through one component
and leave through the other. If the intervening path can be homotoped
into the gluing face within the ambient piece, with endpoints fixed,
the two crossings can cancel after gluing, so such a loop may become
contractible in the ambient manifold. We must rule this out for paths
joining different intersection components. The following criterion
makes this path condition precise.

\begin{lem}[Boundary injectivity for multiple boundary identifications]
\label{lem:boundary}
Let $Z$ be obtained by gluing two connected manifolds $Z_1,Z_2$
along disjoint connected boundary faces $Y_1,\ldots,Y_r$, each
$\pi_1$-injective in both manifolds. Let $B$ be a connected union
of the remaining boundary faces. Suppose each connected piece
$A\subset Z_i$ of $B$ before gluing is $\pi_1$-injective in $Z_i$,
and each component of $A\cap Y_j$ is $\pi_1$-injective in $Y_j$.
For each such $A$ and each $Y_j$ meeting it, assume the following
path condition: if a path in $A$ has endpoints in $A\cap Y_j$
and can be homotoped into $Y_j$ within $Z_i$, fixing its endpoints,
then it can already be homotoped into $A\cap Y_j$ within $A$,
again fixing its endpoints  $(\star)$.
Under these assumptions, $\pi_1(B)\to\pi_1(Z)$ is injective.
\end{lem}

\begin{proof}
Van Kampen's theorem and the normal form theorem for $\pi_1$
of glued spaces~\cite[Theorem~11]{Serre} give injections
$\pi_1(Z_i)\hookrightarrow\pi_1(Z)$.
Let $\alpha$ be a loop in $B$ which is nullhomotopic in $Z$.
Within its homotopy class in $B$, choose $\alpha$ transverse to
the gluing faces and minimizing the total number of intersection
points with them.

Suppose this number is positive.
The normal form theorem says that a nullhomotopic loop must
have two consecutive crossings of the same $Y_j$ in opposite
directions, with the intervening arc homotopic into $Y_j$
within one $Z_i$, fixing endpoints.
This arc lies in one boundary piece $A\subset Z_i$.
By condition $(\star)$, it can also
be homotoped within $A$ into $A\cap Y_j$, fixing endpoints.
The resulting path lies in a single component of $A\cap Y_j$.
A small push across that component, within $B$, then removes
both intersections, contradicting minimality.\footnote{This is a $1+3$ dimensional
analogue in a $4$--manifold of $1+1$ dimensional bigon removal for curves on
a surface. Here the disk traced by the homotopy need not be
embedded; the path condition ensures that the move can be
made within $B$.}

Thus $\alpha$ has no intersections with the gluing faces and
lies in a single piece $A$. Since $\alpha$ is nullhomotopic
in $Z$, and since the composition
$\pi_1(A)\to\pi_1(Z_i)\to\pi_1(Z)$ is injective,
$\alpha$ must already be nullhomotopic in $A$, and hence in $B$.
\end{proof}

\begin{figure}[tbp]
\centering
\begin{tikzpicture}[x=.82cm,y=.82cm,font=\small,
  remaining/.style={draw=blue!65!black,line width=1.5pt},
  gluing/.style={draw=orange!85!black,line width=1.8pt},
  intersection/.style={circle,fill=violet!80!black,draw=white,
    line width=.3pt,inner sep=2.2pt}]
  \filldraw[fill=gray!6,draw=gray!55,rounded corners=4pt]
    (0,0) rectangle (4,3);
  \filldraw[fill=gray!6,draw=gray!55,rounded corners=4pt]
    (7,0) rectangle (11,3);
  \draw[remaining,rounded corners=4pt]
    (4,2.85)--(3.85,3)--(.15,3)--(0,2.85)--(0,.15)
    --(.15,0)--(3.85,0)--(4,.15);
  \draw[remaining,rounded corners=4pt]
    (7,2.85)--(7.15,3)--(10.85,3)--(11,2.85)--(11,.15)
    --(10.85,0)--(7.15,0)--(7,.15);
  \draw[remaining] (4,2)--(4,1);
  \draw[remaining] (7,2)--(7,1);
  \draw[gluing] (4,2.85)--(4,2);
  \draw[gluing] (7,2.85)--(7,2);
  \draw[gluing] (4,1)--(4,.15);
  \draw[gluing] (7,1)--(7,.15);
  \draw[orange!85!black,dashed,<->,>=Stealth]
    (4.15,2.4)--node[above,fill=white,inner sep=2pt] {$Y_1$}(6.85,2.4);
  \draw[orange!85!black,dashed,<->,>=Stealth]
    (4.15,.6)--node[above,fill=white,inner sep=2pt] {$Y_2$}(6.85,.6);
  \node at (2,1.5) {$Z_1$};
  \node at (9,1.5) {$Z_2$};
  \node[blue!65!black,left] at (-.12,1.5) {$A_1$};
  \node[blue!65!black,right] at (11.12,1.5) {$A_2$};
  \node[blue!65!black,left] at (3.88,1.5) {$A_1$};
  \node[blue!65!black,right] at (7.12,1.5) {$A_2$};
  \node[intersection] at (4,2.85) {};
  \node[intersection] at (7,2.85) {};
  \node[intersection] at (4,2) {};
  \node[intersection] at (7,2) {};
  \node[intersection] at (4,1) {};
  \node[intersection] at (7,1) {};
  \node[intersection] at (4,.15) {};
  \node[intersection] at (7,.15) {};
  \node[violet!80!black,right] at (4.12,2.96) {$C_1$};
  \node[violet!80!black,left] at (6.88,2.96) {$C_1$};
  \node[violet!80!black,right] at (4.12,1.93) {$C_2$};
  \node[violet!80!black,left] at (6.88,1.93) {$C_2$};
  \node[violet!80!black,right] at (4.12,1.07) {$C_3$};
  \node[violet!80!black,left] at (6.88,1.07) {$C_3$};
  \node[violet!80!black,right] at (4.12,.02) {$C_4$};
  \node[violet!80!black,left] at (6.88,.02) {$C_4$};
  \draw[remaining] (.1,-.85)--(.7,-.85);
  \node[right] at (.8,-.85) {remaining boundary};
  \draw[gluing] (4.65,-.85)--(5.25,-.85);
  \node[right] at (5.35,-.85) {gluing faces};
  \node[intersection] at (8.45,-.85) {};
  \node[right] at (8.65,-.85) {intersection components};
\end{tikzpicture}
\caption{Schematic section for Lemma~\ref{lem:boundary}, with $B$ formed
from $A_1,A_2$ after the boundary identifications.}
\label{fig:boundarygluings}
\end{figure}
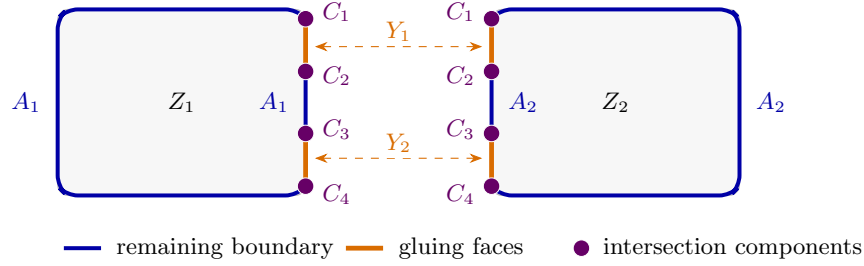
\begin{rmk}\label{rmk:boundarycriterion}
Once the hypotheses have been checked in both pieces,
we may compose a gluing map along $Y$ with a self-diffeomorphism that is the identity near $\partial Y$.
This changes how the interiors of the two copies of $Y$ are
identified, while their boundary collars are identified as before.
(These hypotheses in the lemma concern the individual pieces and their
boundary inclusions, which are unchanged by this modification.)
Thus Lemma~\ref{lem:boundary} and Theorem~\ref{thm:whitehead}
still give asphericity of the resulting space and injectivity of the remaining boundary,
whenever the pieces and gluing faces are aspherical.
The torus surgeries of multiplicity $1$ described above are regluings of this form.
\end{rmk}

\section{The constructions}\label{sec:constructions}

We now prove Theorem~\ref{thm:families}. The main ideas are all
illustrated by the FPS construction. The AP construction is handled
in essentially the same way, but its more involved model requires
us to describe the pieces and their boundary identifications in
greater detail.

We work with the original surgery models of~\cite{FPS,AP} but do not use the global fundamental group
calculations therein. Kirk's local complement calculation~\cite{Kirk}
supplies the end maps and peripheral data used in
Lemma~\ref{lem:multiplicity}. The remaining group calculations follow
from our bundle descriptions and the gluing lemmas of
Section~\ref{sec:gluing}.

\subsection{The Fintushel--Park--Stern construction}\label{sec:fps}

We begin with surgeries on $\Sigma_2\x\Sigma_2$. Cutting each
factor into punctured tori gives the local models of
Section~\ref{sec:gluing} which we refer to as quadrants below for obvious reasons.

We use the curves in Figure~\ref{fig:fps}. The pairs $(T_i,\lambda_i)$
of surgery tori and chosen longitudes from~\cite{FPS}, in the
order $i=1,\ldots,8$, are
\begin{equation}\label{eq:fpslink}
\begin{array}{@{}l@{\qquad}l@{\qquad}l@{\qquad}l@{}}
(a_1'\x c_1',\,a_1'),&
(b_1'\x c_1'',\,b_1'),&
(a_2'\x c_2',\,a_2'),&
(b_2'\x c_2'',\,b_2'),\\[5pt]
(a_2'\x c_1',\,c_1'),&
(a_2''\x d_1',\,d_1'),&
(a_1'\x c_2',\,c_2'),&
(a_1''\x d_2',\,d_2').
\end{array}
\end{equation}
Let $F_1$ denote the symplectic manifold obtained by the eight
Luttinger surgeries with surgery coefficients
$(-1,-1,-1,-1,1,1,1,1)$. For $m\ge1$, let $F_m$ be obtained by
changing only the eighth slope, from $\mu_8+\lambda_8$ to
$m\mu_8+\lambda_8$. Thus $m$ is the meridian exponent denoted $n+1$
in~\cite{FPS}; for $m>1$, this surgery is not Luttinger.
Choose a circle in each $\Sigma_2$ factor separating its two
handles, disjoint from the surgery curves and their parallel
copies. These circles divide the first factor into $P_1,P_2$
and the second into $Q_1,Q_2$, each a torus with one boundary
component, and hand us the four quadrants we'll work with.
For $F_1$, the surgered quadrants have the bundle descriptions in
Table~\ref{tab:quadrants}.

\begin{table}[htbp]
\caption{The surgered quadrants for $F_1$.}\label{tab:quadrants}
\centering
\renewcommand{\arraystretch}{1.2}
\begin{tabular*}{\textwidth}{@{}l@{\extracolsep{\fill}}cccc@{}}
quadrant&$P_1\x Q_1$&$P_1\x Q_2$&$P_2\x Q_1$&$P_2\x Q_2$\\[7pt]
surgery tori&$T_1,T_2$&$T_7,T_8$&$T_5,T_6$&$T_3,T_4$\\[7pt]
(fiber, base)&$(P_1,Q_1)$&$(Q_2,P_1)$&$(Q_1,P_2)$&$(P_2,Q_2)$
\end{tabular*}
\end{table}

In each quadrant, both chosen longitudes lie in the indicated fiber
and the two base curves are parallel. In the quadrants $P_1\x Q_2$
and $P_2\x Q_1$, we have simply exchanged the roles of the two
factors.

\begin{figure}[htbp]
\centering
\resizebox{.88\textwidth}{!}{
\begin{tikzpicture}[x=1bp,y=-1bp,line cap=butt,line join=miter,
  every node/.style={inner sep=0pt,outer sep=0pt,font=\fontsize{7}{8.4}\selectfont}]
  \path[use as bounding box] (0,0) rectangle (360,173);
  \clip (0,0) rectangle (360,173);
  \path[draw=black,line width=0.45bp]
    (61.06,12.996) -- (109.66,12.996) -- (158.26,61.596) -- (158.26,110.197) -- (109.66,158.796)
    (61.06,158.796) -- (109.66,158.796)
    (61.06,158.796) -- (12.46,110.197) -- (12.46,61.596)
    (37.66,36.396) -- (61.06,12.996)
    (12.46,61.596) -- (37.66,36.396)
    (12.46,93.996) -- (77.26,158.796)
    (93.46,12.996) -- (158.26,77.796)
    (28.66,45.396) -- (28.66,126.396)
    (142.06,45.396) -- (142.06,126.396)
    (12.46,77.796) -- (93.46,158.796);
  \path[draw=black,line width=0.45bp,dash pattern=on 2.7bp off 2.7bp]
    (12.46,61.596) -- (109.66,158.796)
    (61.06,12.996) -- (158.26,110.197);
  \path[draw=black,line width=0.45bp]
    (77.26,12.996) -- (158.26,93.996);
  \path[fill=black]
    (111.235,12.996)
    .. controls (111.235,13.743) and (110.63,14.346) .. (109.885,14.346)
    .. controls (109.139,14.346) and (108.535,13.743) .. (108.535,12.996)
    .. controls (108.535,12.251) and (109.139,11.646) .. (109.885,11.646)
    .. controls (110.63,11.646) and (111.235,12.251) .. (111.235,12.996);
  \path[fill=black]
    (62.41,158.796)
    .. controls (62.41,159.543) and (61.805,160.146) .. (61.06,160.146)
    .. controls (60.314,160.146) and (59.71,159.543) .. (59.71,158.796)
    .. controls (59.71,158.051) and (60.314,157.446) .. (61.06,157.446)
    .. controls (61.805,157.446) and (62.41,158.051) .. (62.41,158.796);
  \path[fill=black]
    (13.81,110.197)
    .. controls (13.81,110.943) and (13.205,111.547) .. (12.46,111.547)
    .. controls (11.714,111.547) and (11.11,110.943) .. (11.11,110.197)
    .. controls (11.11,109.451) and (11.714,108.847) .. (12.46,108.847)
    .. controls (13.205,108.847) and (13.81,109.451) .. (13.81,110.197);
  \path[fill=black]
    (13.585,61.596)
    .. controls (13.585,62.343) and (12.98,62.946) .. (12.235,62.946)
    .. controls (11.489,62.946) and (10.885,62.343) .. (10.885,61.596)
    .. controls (10.885,60.851) and (11.489,60.246) .. (12.235,60.246)
    .. controls (12.98,60.246) and (13.585,60.851) .. (13.585,61.596);
  \path[fill=black]
    (62.635,13.221)
    .. controls (62.635,13.968) and (62.03,14.571) .. (61.285,14.571)
    .. controls (60.539,14.571) and (59.935,13.968) .. (59.935,13.221)
    .. controls (59.935,12.476) and (60.539,11.871) .. (61.285,11.871)
    .. controls (62.03,11.871) and (62.635,12.476) .. (62.635,13.221);
  \path[fill=black]
    (159.61,110.422)
    .. controls (159.61,111.168) and (159.005,111.772) .. (158.26,111.772)
    .. controls (157.514,111.772) and (156.91,111.168) .. (156.91,110.422)
    .. controls (156.91,109.676) and (157.514,109.072) .. (158.26,109.072)
    .. controls (159.005,109.072) and (159.61,109.676) .. (159.61,110.422);
  \path[fill=black]
    (159.61,61.596)
    .. controls (159.61,62.343) and (159.005,62.946) .. (158.26,62.946)
    .. controls (157.514,62.946) and (156.91,62.343) .. (156.91,61.596)
    .. controls (156.91,60.851) and (157.514,60.246) .. (158.26,60.246)
    .. controls (159.005,60.246) and (159.61,60.851) .. (159.61,61.596);
  \path[fill=black]
    (111.01,158.796)
    .. controls (111.01,159.543) and (110.405,160.146) .. (109.66,160.146)
    .. controls (108.914,160.146) and (108.31,159.543) .. (108.31,158.796)
    .. controls (108.31,158.051) and (108.914,157.446) .. (109.66,157.446)
    .. controls (110.405,157.446) and (111.01,158.051) .. (111.01,158.796);
  \path[fill=black]
    (85.071,158.796) -- (86.057,160.418) -- (86.002,160.454) -- (83.358,159.404)
    .. controls (82.437,159.202) and (81.516,159) .. (80.595,158.796)
    .. controls (81.516,158.594) and (82.437,158.391) .. (83.358,158.189) -- (86.002,157.138) -- (86.057,157.166) -- (85.071,158.796);
  \path[draw=black,line width=0.45bp]
    (85.36,158.796)
    .. controls (84.897,158.796) and (84.612,158.796) .. (84.612,158.796);
  \path[fill=black]
    (85.071,158.796) -- (86.057,160.418) -- (86.002,160.454) -- (83.358,159.404)
    .. controls (82.437,159.202) and (81.516,159) .. (80.595,158.796)
    .. controls (81.516,158.594) and (82.437,158.391) .. (83.358,158.189) -- (86.002,157.138) -- (86.057,157.166) -- (85.071,158.796);
  \path[fill=black]
    (84.901,12.996) -- (83.916,11.375) -- (83.97,11.339) -- (86.614,12.389)
    .. controls (87.535,12.591) and (88.457,12.794) .. (89.378,12.996)
    .. controls (88.457,13.2) and (87.535,13.402) .. (86.614,13.604) -- (83.97,14.655) -- (83.916,14.627) -- (84.901,12.996);
  \path[draw=black,line width=0.45bp]
    (84.612,12.996)
    .. controls (85.075,12.996) and (85.36,12.996) .. (85.36,12.996);
  \path[fill=black]
    (84.901,12.996) -- (83.916,11.375) -- (83.97,11.339) -- (86.614,12.389)
    .. controls (87.535,12.591) and (88.457,12.794) .. (89.378,12.996)
    .. controls (88.457,13.2) and (87.535,13.402) .. (86.614,13.604) -- (83.97,14.655) -- (83.916,14.627) -- (84.901,12.996);
  \path[fill=black]
    (37.084,36.972) -- (38.928,37.421) -- (38.915,37.486) -- (36.303,38.614)
    .. controls (35.509,39.122) and (34.714,39.63) .. (33.919,40.137)
    .. controls (34.427,39.343) and (34.935,38.548) .. (35.443,37.753) -- (36.57,35.141) -- (36.629,35.122) -- (37.084,36.972);
  \path[draw=black,line width=0.45bp]
    (37.92,36.137)
    .. controls (37.202,36.855) and (36.76,37.296) .. (36.76,37.296);
  \path[fill=black]
    (37.084,36.972) -- (38.928,37.421) -- (38.915,37.486) -- (36.303,38.614)
    .. controls (35.509,39.122) and (34.714,39.63) .. (33.919,40.137)
    .. controls (34.427,39.343) and (34.935,38.548) .. (35.443,37.753) -- (36.57,35.141) -- (36.629,35.122) -- (37.084,36.972);
  \path[fill=black]
    (130.114,33.451) -- (130.564,31.607) -- (130.629,31.62) -- (131.755,34.232)
    .. controls (132.264,35.027) and (132.771,35.821) .. (133.28,36.616)
    .. controls (132.485,36.108) and (131.69,35.6) .. (130.895,35.092) -- (128.284,33.965) -- (128.264,33.907) -- (130.114,33.451);
  \path[draw=black,line width=0.45bp]
    (129.91,33.246)
    .. controls (130.238,33.574) and (130.439,33.775) .. (130.439,33.775);
  \path[fill=black]
    (130.114,33.451) -- (130.564,31.607) -- (130.629,31.62) -- (131.755,34.232)
    .. controls (132.264,35.027) and (132.771,35.821) .. (133.28,36.616)
    .. controls (132.485,36.108) and (131.69,35.6) .. (130.895,35.092) -- (128.284,33.965) -- (128.264,33.907) -- (130.114,33.451);
  \path[fill=black]
    (37.085,134.821) -- (36.636,136.664) -- (36.571,136.652) -- (35.443,134.04)
    .. controls (34.934,133.246) and (34.426,132.452) .. (33.917,131.657)
    .. controls (34.713,132.164) and (35.507,132.672) .. (36.302,133.18) -- (38.915,134.305) -- (38.934,134.364) -- (37.085,134.821);
  \path[draw=black,line width=0.45bp]
    (37.289,135.026)
    .. controls (36.962,134.698) and (36.76,134.496) .. (36.76,134.496);
  \path[fill=black]
    (37.085,134.821) -- (36.636,136.664) -- (36.571,136.652) -- (35.443,134.04)
    .. controls (34.934,133.246) and (34.426,132.452) .. (33.917,131.657)
    .. controls (34.713,132.164) and (35.507,132.672) .. (36.302,133.18) -- (38.915,134.305) -- (38.934,134.364) -- (37.085,134.821);
  \path[fill=black]
    (12.46,86.186) -- (14.081,85.2) -- (14.118,85.255) -- (13.068,87.898)
    .. controls (12.866,88.82) and (12.663,89.741) .. (12.46,90.662)
    .. controls (12.257,89.741) and (12.055,88.82) .. (11.852,87.898) -- (10.802,85.255) -- (10.83,85.2) -- (12.46,86.186);
  \path[draw=black,line width=0.45bp]
    (12.46,85.896)
    .. controls (12.46,86.359) and (12.46,86.645) .. (12.46,86.645);
  \path[fill=black]
    (12.46,86.186) -- (14.081,85.2) -- (14.118,85.255) -- (13.068,87.898)
    .. controls (12.866,88.82) and (12.663,89.741) .. (12.46,90.662)
    .. controls (12.257,89.741) and (12.055,88.82) .. (11.852,87.898) -- (10.802,85.255) -- (10.83,85.2) -- (12.46,86.186);
  \path[fill=black]
    (158.26,85.607) -- (156.639,86.593) -- (156.602,86.538) -- (157.652,83.895)
    .. controls (157.855,82.973) and (158.057,82.052) .. (158.26,81.131)
    .. controls (158.463,82.052) and (158.665,82.973) .. (158.868,83.895) -- (159.918,86.538) -- (159.89,86.593) -- (158.26,85.607);
  \path[draw=black,line width=0.45bp]
    (158.26,85.896)
    .. controls (158.26,85.434) and (158.26,85.148) .. (158.26,85.148);
  \path[fill=black]
    (158.26,85.607) -- (156.639,86.593) -- (156.602,86.538) -- (157.652,83.895)
    .. controls (157.855,82.973) and (158.057,82.052) .. (158.26,81.131)
    .. controls (158.463,82.052) and (158.665,82.973) .. (158.868,83.895) -- (159.918,86.538) -- (159.89,86.593) -- (158.26,85.607);
  \path[fill=black]
    (133.636,134.821) -- (131.792,134.372) -- (131.805,134.307) -- (134.417,133.18)
    .. controls (135.211,132.672) and (136.006,132.163) .. (136.801,131.655)
    .. controls (136.293,132.45) and (135.785,133.245) .. (135.277,134.039) -- (134.15,136.652) -- (134.091,136.671) -- (133.636,134.821);
  \path[draw=black,line width=0.45bp]
    (133.431,135.026)
    .. controls (133.759,134.698) and (133.96,134.496) .. (133.96,134.496);
  \path[fill=black]
    (133.636,134.821) -- (131.792,134.372) -- (131.805,134.307) -- (134.417,133.18)
    .. controls (135.211,132.672) and (136.006,132.163) .. (136.801,131.655)
    .. controls (136.293,132.45) and (135.785,133.245) .. (135.277,134.039) -- (134.15,136.652) -- (134.091,136.671) -- (133.636,134.821);
  \path[draw=black,line width=0.45bp]
    (246.46,14.037) -- (295.06,14.037) -- (343.66,62.637) -- (343.66,111.237) -- (295.06,159.837)
    (246.46,159.837) -- (295.06,159.837)
    (246.46,159.837) -- (197.86,111.237) -- (197.86,62.637)
    (223.06,37.437) -- (246.46,14.037)
    (197.86,62.637) -- (223.06,37.437)
    (197.86,95.037) -- (262.66,159.837)
    (278.86,14.037) -- (343.66,78.837)
    (214.06,46.437) -- (214.06,127.437)
    (327.46,46.437) -- (327.46,127.437)
    (197.86,78.837) -- (278.86,159.837);
  \path[draw=black,line width=0.45bp,dash pattern=on 2.7bp off 2.7bp]
    (197.86,62.637) -- (295.06,159.837)
    (246.46,14.037) -- (343.66,111.237);
  \path[draw=black,line width=0.45bp]
    (262.66,14.037) -- (343.66,95.037);
  \path[fill=black]
    (296.635,14.037)
    .. controls (296.635,14.782) and (296.032,15.387) .. (295.285,15.387)
    .. controls (294.54,15.387) and (293.935,14.782) .. (293.935,14.037)
    .. controls (293.935,13.291) and (294.54,12.687) .. (295.285,12.687)
    .. controls (296.032,12.687) and (296.635,13.291) .. (296.635,14.037);
  \path[fill=black]
    (247.81,159.837)
    .. controls (247.81,160.582) and (247.207,161.187) .. (246.46,161.187)
    .. controls (245.715,161.187) and (245.11,160.582) .. (245.11,159.837)
    .. controls (245.11,159.09) and (245.715,158.487) .. (246.46,158.487)
    .. controls (247.207,158.487) and (247.81,159.09) .. (247.81,159.837);
  \path[fill=black]
    (199.21,111.237)
    .. controls (199.21,111.983) and (198.607,112.587) .. (197.86,112.587)
    .. controls (197.115,112.587) and (196.51,111.983) .. (196.51,111.237)
    .. controls (196.51,110.491) and (197.115,109.887) .. (197.86,109.887)
    .. controls (198.607,109.887) and (199.21,110.491) .. (199.21,111.237);
  \path[fill=black]
    (198.985,62.637)
    .. controls (198.985,63.382) and (198.382,63.987) .. (197.635,63.987)
    .. controls (196.89,63.987) and (196.285,63.382) .. (196.285,62.637)
    .. controls (196.285,61.891) and (196.89,61.287) .. (197.635,61.287)
    .. controls (198.382,61.287) and (198.985,61.891) .. (198.985,62.637);
  \path[fill=black]
    (248.035,14.262)
    .. controls (248.035,15.007) and (247.432,15.612) .. (246.685,15.612)
    .. controls (245.94,15.612) and (245.335,15.007) .. (245.335,14.262)
    .. controls (245.335,13.516) and (245.94,12.912) .. (246.685,12.912)
    .. controls (247.432,12.912) and (248.035,13.516) .. (248.035,14.262);
  \path[fill=black]
    (345.01,111.462)
    .. controls (345.01,112.208) and (344.407,112.812) .. (343.66,112.812)
    .. controls (342.915,112.812) and (342.31,112.208) .. (342.31,111.462)
    .. controls (342.31,110.716) and (342.915,110.112) .. (343.66,110.112)
    .. controls (344.407,110.112) and (345.01,110.716) .. (345.01,111.462);
  \path[fill=black]
    (345.01,62.637)
    .. controls (345.01,63.382) and (344.407,63.987) .. (343.66,63.987)
    .. controls (342.915,63.987) and (342.31,63.382) .. (342.31,62.637)
    .. controls (342.31,61.891) and (342.915,61.287) .. (343.66,61.287)
    .. controls (344.407,61.287) and (345.01,61.891) .. (345.01,62.637);
  \path[fill=black]
    (296.41,159.837)
    .. controls (296.41,160.582) and (295.807,161.187) .. (295.06,161.187)
    .. controls (294.315,161.187) and (293.71,160.582) .. (293.71,159.837)
    .. controls (293.71,159.09) and (294.315,158.487) .. (295.06,158.487)
    .. controls (295.807,158.487) and (296.41,159.09) .. (296.41,159.837);
  \path[fill=black]
    (270.471,159.837) -- (271.457,161.458) -- (271.402,161.494) -- (268.758,160.445)
    .. controls (267.837,160.242) and (266.916,160.04) .. (265.995,159.837)
    .. controls (266.916,159.634) and (267.837,159.431) .. (268.758,159.229) -- (271.402,158.178) -- (271.457,158.207) -- (270.471,159.837);
  \path[draw=black,line width=0.45bp]
    (270.76,159.837)
    .. controls (270.297,159.837) and (270.012,159.837) .. (270.012,159.837);
  \path[fill=black]
    (270.471,159.837) -- (271.457,161.458) -- (271.402,161.494) -- (268.758,160.445)
    .. controls (267.837,160.242) and (266.916,160.04) .. (265.995,159.837)
    .. controls (266.916,159.634) and (267.837,159.431) .. (268.758,159.229) -- (271.402,158.178) -- (271.457,158.207) -- (270.471,159.837);
  \path[fill=black]
    (270.302,14.037) -- (269.316,12.416) -- (269.371,12.379) -- (272.014,13.429)
    .. controls (272.935,13.632) and (273.856,13.834) .. (274.778,14.037)
    .. controls (273.856,14.239) and (272.935,14.442) .. (272.014,14.645) -- (269.371,15.695) -- (269.316,15.667) -- (270.302,14.037);
  \path[draw=black,line width=0.45bp]
    (270.012,14.037)
    .. controls (270.475,14.037) and (270.76,14.037) .. (270.76,14.037);
  \path[fill=black]
    (270.302,14.037) -- (269.316,12.416) -- (269.371,12.379) -- (272.014,13.429)
    .. controls (272.935,13.632) and (273.856,13.834) .. (274.778,14.037)
    .. controls (273.856,14.239) and (272.935,14.442) .. (272.014,14.645) -- (269.371,15.695) -- (269.316,15.667) -- (270.302,14.037);
  \path[fill=black]
    (222.484,38.012) -- (224.329,38.461) -- (224.316,38.526) -- (221.704,39.653)
    .. controls (220.909,40.161) and (220.115,40.67) .. (219.32,41.178)
    .. controls (219.828,40.384) and (220.336,39.588) .. (220.844,38.794) -- (221.97,36.181) -- (222.029,36.162) -- (222.484,38.012);
  \path[draw=black,line width=0.45bp]
    (223.32,37.177)
    .. controls (222.602,37.895) and (222.16,38.337) .. (222.16,38.337);
  \path[fill=black]
    (222.484,38.012) -- (224.329,38.461) -- (224.316,38.526) -- (221.704,39.653)
    .. controls (220.909,40.161) and (220.115,40.67) .. (219.32,41.178)
    .. controls (219.828,40.384) and (220.336,39.588) .. (220.844,38.794) -- (221.97,36.181) -- (222.029,36.162) -- (222.484,38.012);
  \path[fill=black]
    (315.515,34.491) -- (315.964,32.648) -- (316.029,32.661) -- (317.156,35.272)
    .. controls (317.664,36.067) and (318.172,36.862) .. (318.68,37.657)
    .. controls (317.886,37.148) and (317.091,36.64) .. (316.296,36.132) -- (313.684,35.005) -- (313.665,34.946) -- (315.515,34.491);
  \path[draw=black,line width=0.45bp]
    (315.31,34.287)
    .. controls (315.638,34.614) and (315.839,34.816) .. (315.839,34.816);
  \path[fill=black]
    (315.515,34.491) -- (315.964,32.648) -- (316.029,32.661) -- (317.156,35.272)
    .. controls (317.664,36.067) and (318.172,36.862) .. (318.68,37.657)
    .. controls (317.886,37.148) and (317.091,36.64) .. (316.296,36.132) -- (313.684,35.005) -- (313.665,34.946) -- (315.515,34.491);
  \path[fill=black]
    (222.485,135.861) -- (222.035,137.705) -- (221.971,137.692) -- (220.844,135.08)
    .. controls (220.336,134.285) and (219.827,133.491) .. (219.32,132.695)
    .. controls (220.114,133.203) and (220.909,133.712) .. (221.703,134.22) -- (224.316,135.347) -- (224.335,135.405) -- (222.485,135.861);
  \path[draw=black,line width=0.45bp]
    (222.689,136.065)
    .. controls (222.362,135.738) and (222.16,135.537) .. (222.16,135.537);
  \path[fill=black]
    (222.485,135.861) -- (222.035,137.705) -- (221.971,137.692) -- (220.844,135.08)
    .. controls (220.336,134.285) and (219.827,133.491) .. (219.32,132.695)
    .. controls (220.114,133.203) and (220.909,133.712) .. (221.703,134.22) -- (224.316,135.347) -- (224.335,135.405) -- (222.485,135.861);
  \path[fill=black]
    (197.86,87.225) -- (199.482,86.24) -- (199.518,86.295) -- (198.468,88.939)
    .. controls (198.266,89.86) and (198.063,90.782) .. (197.86,91.702)
    .. controls (197.657,90.782) and (197.455,89.86) .. (197.252,88.939) -- (196.202,86.295) -- (196.23,86.24) -- (197.86,87.225);
  \path[draw=black,line width=0.45bp]
    (197.86,86.937)
    .. controls (197.86,87.4) and (197.86,87.685) .. (197.86,87.685);
  \path[fill=black]
    (197.86,87.225) -- (199.482,86.24) -- (199.518,86.295) -- (198.468,88.939)
    .. controls (198.266,89.86) and (198.063,90.782) .. (197.86,91.702)
    .. controls (197.657,90.782) and (197.455,89.86) .. (197.252,88.939) -- (196.202,86.295) -- (196.23,86.24) -- (197.86,87.225);
  \path[fill=black]
    (343.66,86.648) -- (342.039,87.634) -- (342.002,87.578) -- (343.052,84.934)
    .. controls (343.255,84.014) and (343.457,83.093) .. (343.66,82.171)
    .. controls (343.863,83.093) and (344.065,84.014) .. (344.268,84.934) -- (345.319,87.578) -- (345.291,87.634) -- (343.66,86.648);
  \path[draw=black,line width=0.45bp]
    (343.66,86.937)
    .. controls (343.66,86.473) and (343.66,86.189) .. (343.66,86.189);
  \path[fill=black]
    (343.66,86.648) -- (342.039,87.634) -- (342.002,87.578) -- (343.052,84.934)
    .. controls (343.255,84.014) and (343.457,83.093) .. (343.66,82.171)
    .. controls (343.863,83.093) and (344.065,84.014) .. (344.268,84.934) -- (345.319,87.578) -- (345.291,87.634) -- (343.66,86.648);
  \path[fill=black]
    (319.036,135.861) -- (317.192,135.412) -- (317.205,135.347) -- (319.817,134.22)
    .. controls (320.612,133.712) and (321.406,133.203) .. (322.201,132.695)
    .. controls (321.694,133.49) and (321.185,134.285) .. (320.677,135.08) -- (319.55,137.692) -- (319.491,137.711) -- (319.036,135.861);
  \path[draw=black,line width=0.45bp]
    (318.832,136.065)
    .. controls (319.159,135.738) and (319.36,135.537) .. (319.36,135.537);
  \path[fill=black]
    (319.036,135.861) -- (317.192,135.412) -- (317.205,135.347) -- (319.817,134.22)
    .. controls (320.612,133.712) and (321.406,133.203) .. (322.201,132.695)
    .. controls (321.694,133.49) and (321.185,134.285) .. (320.677,135.08) -- (319.55,137.692) -- (319.491,137.711) -- (319.036,135.861);
  \path[draw=black,line width=0.45bp]
    (173.11,81.396) -- (185.26,93.546)
    (185.26,81.396) -- (173.11,93.546);
  \node at (32.418,61.43) {$b_1'$};
  \node at (86.237,165.605) {$b_1$};
  \node at (110.695,54.207) {$a_2''$};
  \node at (54.445,112.888) {$a_1''$};
  \node at (135.018,31.212) {$a_2$};
  \node at (31.653,30.11) {$a_1$};
  \node at (47.538,123.237) {$a_1'$};
  \node at (30.753,141.935) {$a_1$};
  \node at (86.26,7.475) {$b_2$};
  \node at (165.235,86.112) {$b_2$};
  \node at (119.335,45.815) {$a_2'$};
  \node at (5.26,85.73) {$b_1$};
  \node at (140.643,138.987) {$a_2$};
  \node at (136.435,109.625) {$b_2'$};
  \node at (57.438,164.052) {$x$};
  \node at (114.812,10.602) {$x$};
  \node at (56.988,10.602) {$x$};
  \node at (7.488,59.653) {$x$};
  \node at (162.738,58.977) {$x$};
  \node at (8.613,114.102) {$x$};
  \node at (161.838,114.553) {$x$};
  \node at (111.213,164.502) {$x$};
  \node at (217.885,62.465) {$d_1'$};
  \node at (271.637,166.662) {$d_1$};
  \node at (295.847,55.242) {$c_2''$};
  \node at (240.07,113.472) {$c_1''$};
  \node at (320.417,32.247) {$c_2$};
  \node at (217.052,30.853) {$c_1$};
  \node at (232.578,124.138) {$c_1'$};
  \node at (216.377,142.7) {$c_1$};
  \node at (271.66,8.51) {$d_2$};
  \node at (350.185,86.99) {$d_2$};
  \node at (304.735,47.3) {$c_2'$};
  \node at (190.66,86.765) {$d_1$};
  \node at (326.042,139.865) {$c_2$};
  \node at (321.835,110.66) {$d_2'$};
  \node at (242.86,165.74) {$y$};
  \node at (300.235,12.29) {$y$};
  \node at (242.41,12.29) {$y$};
  \node at (192.91,61.34) {$y$};
  \node at (348.16,60.665) {$y$};
  \node at (194.035,115.79) {$y$};
  \node at (347.26,116.24) {$y$};
  \node at (296.635,166.19) {$y$};
\end{tikzpicture}%
}
\caption{The two $\Sigma_2$ factors of the product and the parallel copies of the curves specifying
 the FPS surgery tori~\cite[Figure~2]{FPS}. The curves with subscript $i$ lie in the
 punctured torus $P_i$, respectively $Q_i$, in the quadrant decomposition.}
\label{fig:fps}
\end{figure}
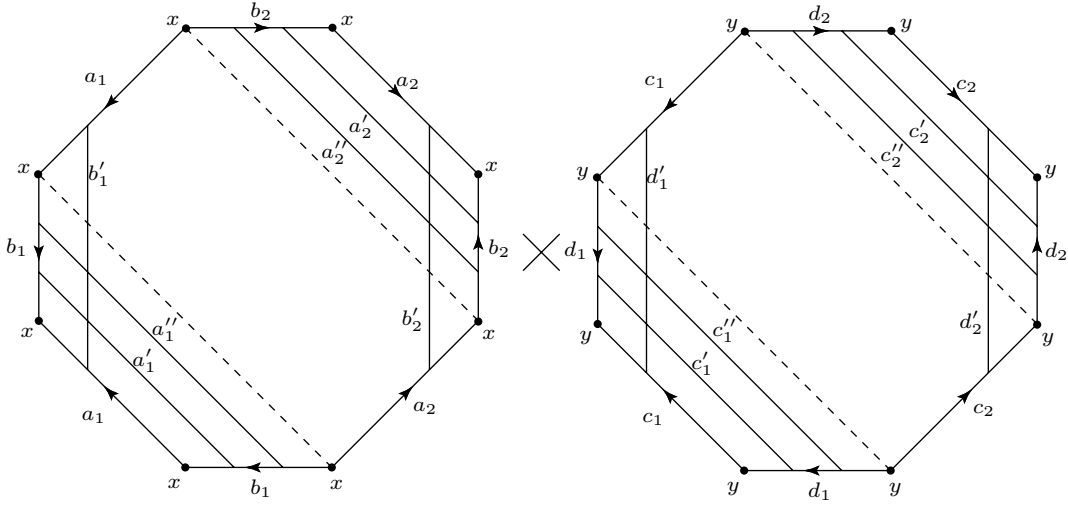

\begin{proof}[The FPS case of Theorem~\ref{thm:families}]
We first treat $F_1$. Let $Z_{ij}$ denote the surgered quadrants.
Lemma~\ref{lem:product}
applies to all four. Form the two rows
\[
R_i=Z_{i1}\cup_{P_i\x S^1}Z_{i2}.
\]
The common product faces are aspherical and inject into the quadrants,
so Whitehead's theorem makes each row aspherical. Its remaining
boundary is
\[
\partial R_i=\partial P_i\x\Sigma_2
=(\partial P_i\x Q_1)\cup_{T^2}(\partial P_i\x Q_2).
\]
The intersection calculation~\eqref{eq:corner}, using the face roles
in Table~\ref{tab:quadrants}, and~\eqref{eq:subamalgam} show that
$\pi_1(\partial R_i)$ injects into $\pi_1(R_i)$. A second application
of Whitehead to
\[
F_1=R_1\cup_{S^1\x\Sigma_2}R_2
\]
proves asphericity. The gluing here is the original product gluing,
which no surgery has changed. In particular, the separating
$S^1\x\Sigma_2$ remains $\pi_1$-injective in the closed manifold.

For $F_m$, only $Z_{12}$ changes with $m$. In this quadrant the fiber is
$Q_2$, the two fiber curves $c_2,d_2$ intersect once, and the
parallel base curves lie in $P_1$. The surgery on $T_7$ remains a
surgery of multiplicity $1$, while that on $T_8$ has slope $m\mu_8+\lambda_8$.
Lemma~\ref{lem:multiplicity} supplies exactly the two face
injections and the corner intersection used above. The same two
gluings therefore prove that $F_m$ is aspherical and that the
separating $S^1\x\Sigma_2$ remains $\pi_1$-injective.

The homology calculation for $F_m$ appears in~\cite{FPS}. We include
these quick arguments to make the note self-contained. Let $W$ be the
complement of the eight original tori in
$Y=\Sigma_2\x\Sigma_2$. Excision gives
\[
H_2(Y;\Z)\longrightarrow\Z^8\longrightarrow H_1(W;\Z)
\longrightarrow H_1(Y;\Z)\longrightarrow0.
\]
The first map records intersection numbers. In the order
of Lagrangian tori in \eqref{eq:fpslink}, dual tori classes are, up to signs,
\[
b_1 \x d_1 ,\ a_1\mathbin{\x}d_1,\ b_2\mathbin{\x}d_2,\ a_2\mathbin{\x}d_2,\
b_2\mathbin{\x}d_1,\ b_2\mathbin{\x}c_1,\ b_1\mathbin{\x}d_2,\ b_1\mathbin{\x}c_2.
\]
Each pairs with just its corresponding surgery class. Thus
$H_1(W;\Z)=\Z^8$, and the meridians vanish in this group. The eight
chosen longitudes represent the basis
$a_1,b_1,a_2,b_2,c_1,d_1,c_2,d_2$. Each longitude still has
coefficient $\pm1$, so $H_1(F_m;\Z)=0$ for every $m$.
The surgeries preserve Euler characteristic $4$ and signature $0$.
The two surviving product surfaces have square zero and intersection
one, and give a hyperbolic basis for $H_2(F_m;\Z)$. Thus $F_m$ has
the integral homology and intersection form of $S^2\x S^2$.
\end{proof}

\begin{rmk}\label{rmk:otherfaces}
We can equally well first glue along the other pair of
product faces. Thus, if the two surgered pieces $P_i\x Q$
of $\Sigma_2\x Q$ satisfy Lemma~\ref{lem:product} or
Lemma~\ref{lem:multiplicity}, the resulting piece is aspherical
and its boundary $\Sigma_2\x\partial Q$ is $\pi_1$-injective.
This follows from Theorem~\ref{thm:whitehead} and~\eqref{eq:subamalgam}.
\end{rmk}

\subsection{The Akhmedov--Park construction}\label{sec:ap}

Here the model, the Hacon--Pardini surface~\cite{HP}, is a
$\Sigma_2$--bundle over $\Sigma_2$ with nontrivial monodromy.
Writing $S=\Sigma_2$, it is the quotient
\[
X=(S\x\Sigma_3)/\langle(\tau,\tau_2)\rangle
\longrightarrow\Sigma_2
\]
where $\tau\colon S\to S$ is an involution with two fixed points and
$\tau_2\colon\Sigma_3\to\Sigma_3$ is a free involution~\cite{AP}.
This bundle has a fairly simple monodromy
\[ [\id, \id] [\tau, \id]= \id \quad\text{in }\Mod(S).
\]
Split the base into $Q_i \cong \Sigma_1^1$, $i=1,2$, so that the bundle over $Q_1$ is trivial; over $Q_2$, the generators
$c=\widetilde c_2$ and $d=d_2$ act by $\tau$ and the identity, respectively,  and the two bundles  are identified along their common boundary $S\x S^1$. 

The six pairs $(T_i,\lambda_i)$ of tori and chosen longitudes
in~\cite{AP}, in the order $i=1,\ldots,6$, are
\begin{equation}\label{eq:aplink}
\begin{array}{@{}l@{\qquad}l@{\qquad}l@{}}
\bigl(q(a_1'\x c_1'),\,a_1'\bigr),&
\bigl(q(b_1'\x c_1''),\,b_1'\bigr),&
\bigl(q(a_2'\x c_1'),\,c_1'\bigr),\\[5pt]
\bigl(q(a_2''\x d_1'),\,d_1'\bigr),&
\bigl(q(b_1''\x d_2'),\,d_2'\bigr),&
\bigl(q((\widetilde b_1\widetilde b_2)\x\widetilde c_2'),\,\beta\bigr).
\end{array}
\end{equation}
Here $q$ is the quotient map and $\beta$ is the curve
in~\cite[Equation~(8)]{AP}. The sixth torus is obtained from the cylinder
$(\widetilde b_1\widetilde b_2)\x\widetilde c_2'$
by identifying its two ends using $\tau$.
We use the notation $M_m^1$, $m\ge1$, from
\cite{AP}, with $p=1$. The symplectic member $M_1^1$ is
obtained by the six Luttinger surgeries with slopes $\mu_i-\lambda_i$.
In $M_m^1$, only the first changes, to $m\mu_1-\lambda_1$.

\begin{proof}[The AP case of Theorem~\ref{thm:families}]
We first treat $M_1^1$. The four surgeries over $Q_1$ give the FPS
local models after splitting $S$ into punctured tori and exchanging
factors in one piece. Lemma~\ref{lem:product} and
Remark~\ref{rmk:otherfaces} show that the surgered piece $E_1$
is aspherical with $\pi_1$-injective boundary $S\x\partial Q_1$.

For the two surgeries over $Q_2$, we decompose the fiber $S$ into
two $\tau$-invariant pairs of pants $P_+,P_-$, as shown in
Figure~\ref{fig:ap-pants}. The cutting curves are
\[
\eta_0=b_1'',\qquad \eta_1=\tau(b_1''),\qquad
\zeta=\widetilde b_1\widetilde b_2.
\]
In the octagon, $\eta_0,\eta_1$ cut off the northwest and southeast
corners, while the two arcs $\zeta_1,\zeta_2$ of $\zeta$ cut off the
other two.
The central region has two pairs of intervals identified, three
boundary circles $\eta_0,\eta_1,\zeta$ and Euler characteristic $-1$.
Its complement is connected through $z_0$ and has the same boundary
circles and Euler characteristic. Thus the two regions give
$P_+,P_-$, each preserved by the half turn $\tau$. It fixes
$z_1,z_0$, respectively, exchanges $\eta_0,\eta_1$ and acts freely
on $\zeta$.

\begin{figure}[htbp]
\centering
\definecolor{curveblue}{RGB}{0,105,170}
\definecolor{curveorange}{RGB}{196,110,20}
\definecolor{pathpurple}{RGB}{117,77,143}
\tikzset{
  edge/.style={line width=.5pt,draw=black!65},
  curve/.style={line width=1pt,line cap=round},
  arrow/.style={postaction={decorate},decoration={markings,
    mark=at position .5 with {\arrow{Stealth[length=4pt,width=3.4pt]}}}}
}
\resizebox{.76\textwidth}{!}{%
\begin{tikzpicture}[x=1.20cm,y=1.20cm,font=\normalsize]
  \coordinate (A) at (-1,2.4);
  \coordinate (B) at (1,2.4);
  \coordinate (C) at (2.4,1);
  \coordinate (D) at (2.4,-1);
  \coordinate (E) at (1,-2.4);
  \coordinate (F) at (-1,-2.4);
  \coordinate (G) at (-2.4,-1);
  \coordinate (H) at (-2.4,1);
  \path[fill=black!2] (A)--(B)--(C)--(D)--(E)--(F)--(G)--(H)--cycle;
  \path[fill=curveblue!7]
    (.15,2.4)--(.75,2.4)--(2.4,.75)--(2.4,.15)
    --(-.15,-2.4)--(-.75,-2.4)--(-2.4,-.75)--(-2.4,-.15)--cycle;
  \draw[edge,arrow] (B)--(A) node[midway,above=5pt] {$a_1$};
  \draw[edge,arrow] (A)--(H) node[midway,above left=3pt] {$b_1$};
  \draw[edge,arrow] (G)--(H) node[midway,left=6pt] {$a_1$};
  \draw[edge,arrow] (F)--(G) node[midway,below left=3pt] {$b_1$};
  \draw[edge,arrow] (F)--(E) node[midway,below=5pt] {$a_2$};
  \draw[edge,arrow] (E)--(D) node[midway,below right=3pt] {$b_2$};
  \draw[edge,arrow] (C)--(D) node[midway,right=6pt] {$a_2$};
  \draw[edge,arrow] (B)--(C) node[midway,above right=3pt] {$b_2$};

  \draw[curve,curveblue] (.15,2.4)--(-2.4,-.15);
  \draw[curve,curveblue] (-.15,-2.4)--(2.4,.15);
  \draw[curve,curveorange] (.75,2.4)--(2.4,.75);
  \draw[curve,curveorange] (-.75,-2.4)--(-2.4,-.75);
  \node[curveblue,fill=black!2,inner sep=1.3pt] at (-1.14,1.39) {$\eta_0$};
  \node[curveblue,fill=black!2,inner sep=1.3pt] at (1.14,-1.39) {$\eta_1$};
  \node[curveorange] at (1.43,1.38) {$\zeta_1$};
  \node[curveorange] at (-1.43,-1.38) {$\zeta_2$};
  \node[black!65] at (-.53,-.63) {$P_+$};
  \node[black!65] at (-1.58,1.51) {$P_-$};
  \foreach \p in {A,B,C,D,E,F,G,H} \fill (\p) circle[radius=1.1pt];
  \node[anchor=south east] at (-1.08,2.46) {$z_0$};
  \fill (0,0) circle[radius=1.4pt];
  \node[anchor=north west] at (.07,-.04) {$z_1$};
  \draw[pathpurple,line width=.65pt,-{Stealth[length=4pt]}]
    (.43,0) arc[start angle=0,end angle=180,radius=.43];
  \node[pathpurple] at (0,.70) {$\tau$};
  \node[anchor=north,font=\small] at (0,-2.99)
    {(a) The fiber $S$.};

  \begin{scope}[shift={(7.2,0)}]
    \path[fill=curveblue!7] (0,0) circle[radius=2.2];
    \fill[white] (-1.1,0) circle[radius=.43];
    \fill[white] (1.1,0) circle[radius=.43];
    \draw[curve,curveorange] (0,0) circle[radius=2.2];
    \draw[curve,curveblue] (-1.1,0) circle[radius=.43];
    \draw[curve,curveblue] (1.1,0) circle[radius=.43];
    \node[curveblue] at (-1.1,-.69) {$\eta_0$};
    \node[curveblue] at (1.1,-.69) {$\eta_1$};
    \node[curveorange] at (1.70,1.77) {$\zeta$};
    \node[black!65] at (0,-1.48) {$P_+$};
    \fill (0,0) circle[radius=1.4pt];
    \node[anchor=north] at (0,-.16) {$z_1$};
    \draw[pathpurple,line width=.65pt,-{Stealth[length=4pt]}]
      (1.78,0) arc[start angle=0,end angle=180,radius=1.78];
    \node[pathpurple] at (0,1.49) {$\tau$};
    \node[anchor=north,font=\small] at (0,-2.99)
      {(b) The pair of pants $P_+$.};
  \end{scope}
  \coordinate (figSW) at (current bounding box.south west);
  \coordinate (figNE) at (current bounding box.north east);
  \path[use as bounding box]
    ([xshift=-5pt,yshift=-5pt]figSW) rectangle
    ([xshift=5pt,yshift=5pt]figNE);
\end{tikzpicture}%
}
\caption{The $\tau$-invariant pants decomposition of the fiber $S$
in the octagon model of~\cite[Figure~2]{AP}, with edges identified as labeled
and all vertices representing $z_0$. The central and corner regions
give $P_+$ and $P_-$. The arcs $\zeta_1,\zeta_2$ join under the edge
identifications to form the single curve $\zeta$. On the right,
$P_+$ is drawn as a planar pair of pants with the half turn exchanging
$\eta_0,\eta_1$.}
\label{fig:ap-pants}
\end{figure}
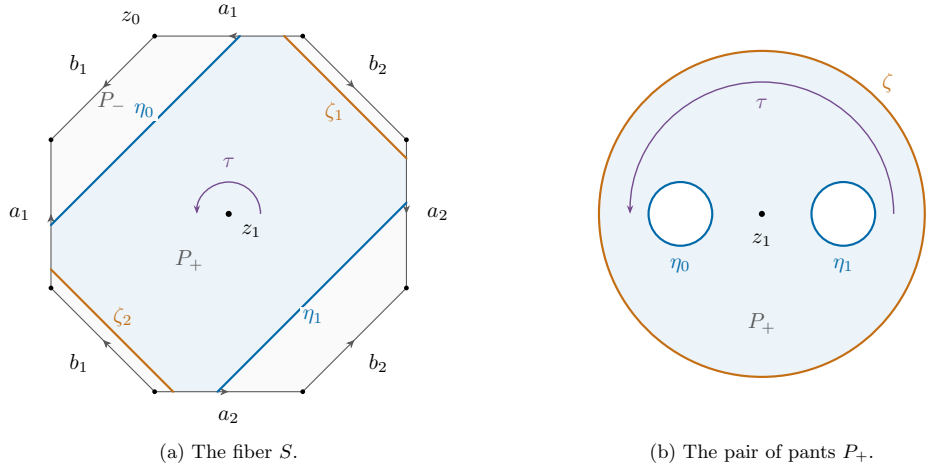

The curves give vertical $3$--manifolds $Y_0,Y_1$ corresponding
to $\zeta$ and $\eta_0\cup\eta_1$; cutting along them gives pants
bundles $Z_+,Z_-$ over $Q_2$. The faces $Y_0,Y_1$ are also bundles
over $Q_2$, with fibers $S^1$ and $S^1\sqcup S^1$, respectively;
$\tau$ exchanges the two components of the latter. These spaces
are aspherical, and the cuff inclusions and bundle homotopy sequences
give $\pi_1$-injective maps $Y_j\to Z_\pm$.

We check that the vertical faces give the prescribed Lagrangian
framings on $T_5\subset\Int Y_1$ and $T_6\subset\Int Y_0$.
The exchanged cuff collars give the product framing for $T_5$.
For $T_6$, choose an invariant collar of $\zeta$ with
$\tau(\theta,v)=(\theta+\tfrac12,v)$. With $s$ along the lifted
base path and $r$ normal to it, the torus neighborhood is
\[
(\theta,s,v,r)\sim(\theta+1,s,v,r)
\sim(\theta+\tfrac12,s+1,v,r).
\]
The longitude $\beta$ first traverses the base path, then
$\widetilde b_2$, halfway around $\zeta$, closing by the second
identification. The normal coordinates $v,r$ descend unchanged;
fixing them gives nearby Lagrangian tori for the product form
in~\cite{AP}, hence the prescribed framing.

The surgeries therefore reglue $Y_0,Y_1$ by maps fixed near their
boundaries, so Theorem~\ref{thm:whitehead} makes the surgered piece
$E_2$ aspherical. To apply Lemma~\ref{lem:boundary} at its outer
boundary, let $A=P_\pm\x\partial Q_2$ be a horizontal face of
$Z_\pm$. The bundle homotopy sequences, using
$\pi_1(\partial Q_2)\hookrightarrow\pi_1(Q_2)$, show that $A$
is $\pi_1$-injective in $Z_\pm$ and that each component of
$A\cap Y_j$ is $\pi_1$-injective in $Y_j$.

Suppose a path $\alpha$ in $A$ is homotopic in $Z_\pm$ to a path
$\gamma$ in $Y_j$, fixing endpoints. Project a homotopy from
$\gamma$ to $\alpha$ to $Q_2$ and lift it to $Y_j\to Q_2$,
starting with $\gamma$ and fixing both endpoints~\cite[Proposition~4.48]{Hatcher}.
The resulting path $\gamma'$ lies over $\partial Q_2$, hence in
a single component of $A\cap Y_j$, and is homotopic to $\alpha$
in $Z_\pm$.
Since $A$ is $\pi_1$-injective, the loop formed by these paths
is nullhomotopic in $A$, so they are homotopic within $A$, fixing
endpoints. This verifies condition $(\star)$ of Lemma~\ref{lem:boundary} in both pants bundles. The regluings preserve these hypotheses by
Remark~\ref{rmk:boundarycriterion}, so Lemma~\ref{lem:boundary}
gives $\pi_1$-injectivity of $S\x\partial Q_2$ in $E_2$.
Theorem~\ref{thm:whitehead} then makes
$M_1^1=E_1\cup_{S\x S^1}E_2$ aspherical, with the separating
boundary still $\pi_1$-injective.

For $M_m^1$, only the first product quadrant of $E_1$ changes.
Its fiber curves $a_1,b_1$ intersect once and the other surgery
has multiplicity $1$, so Lemma~\ref{lem:multiplicity} gives
asphericity and the same face injections and corner intersection.
Thus $E_1$ still has $\pi_1$-injective boundary, and gluing to
the unchanged $E_2$ proves asphericity for every $m\ge1$.

A homology calculation similar to the FPS case gives the integral
homology and intersection form of $S^2\x S^2$ for $M_m^1$~\cite{AP}.
\end{proof}

\section{Luttinger surgery and symplectic asphericity}\label{sec:nonpreservation}

We prove Theorem~\ref{thm:nonpreservation} through some examples we found easiest to demonstrate the failure of the symplectic asphericity under Luttinger surgeries.

\begin{proof}[Proof of Theorem~\ref{thm:nonpreservation}]
Fix $h\ge1$. Let $V\to S^2$ be Matsumoto's genus-$2$ Lefschetz
fibration with six nonseparating and two separating vanishing cycles.  
The untwisted fiber sum
\[
X=V\#_{\Sigma_2}(\Sigma_2\x\Sigma_h)
\]
gives a Lefschetz fibration $f\colon X\to\Sigma_h$ with a section.
We use $a_1,b_1,a_2,b_2$ for the standard $\pi_1$ generators of the fiber and
$c_j,d_j$, for $1\le j\le h$, for the base generators. Our $V$ is
$X_2$ of~\cite{BaykurKorkmaz}; when $h=1$, $c_1,d_1$ correspond
to $a,b$ there. The vanishing cycles give in $\pi_1(V)$ the relations
\[
a_1=a_2^{-1},\qquad b_1=b_2^{-1},\qquad [a_2,b_2]=1,
\]
so $\pi_1(V)\cong\Z^2$, with ordered basis $a_2, b_2$, which we use in the identification 
\[\pi_1(X)\cong\pi_1(V)\x\pi_1(\Sigma_h)
           \cong\Z^2\x\pi_1(\Sigma_h).\]
           
The two coordinate projections of the $\Z^2$ factor descend to
$H_1(X;\Z)$ and give classes
$p,q\in H^1(X;\Z)=\operatorname{Hom}(H_1(X;\Z),\Z)$,
dual to $a_2,b_2$, respectively, and zero on the base generators.
Put $\eta=p\cup q$. The relations above give
$p(a_1)=q(b_1)=-1$ and $p(a_2)=q(b_2)=1$, with all the other
pairings $p(b_i),q(a_i)$ zero. Thus
\[
\langle\eta,[\Sigma_2]\rangle
=\sum_{i=1}^{2}\bigl(p(a_i)q(b_i)-p(b_i)q(a_i)\bigr)
=1+1=2.
\]
At either of the two reducible fibers,
the torus components carry the 
generators $a_1,b_1$ and $a_2,b_2$, respectively. On each component $\eta$ evaluates to~$1$. Since every irreducible fiber represents the regular fiber class, we conclude that $\eta$ is positive on every irreducible fiber component.

We can now use $\eta$ to construct a Gompf--Thurston symplectic form $\omega$ with
\[
[\omega]=\eta+\kappa f^*[\sigma],
\]
where $\sigma$ is an area form on the base $\Sigma_h$ and $\kappa$ is
sufficiently large~\cite[Theorem~2.7]{Gompf}. We choose $\omega$ to agree with a product form
near $T=a_2'\x d_1'\subset\Sigma_2\x\Sigma_h^1$, where the primes
denote parallel copies of the corresponding curves. This is possible
because $\eta$ restricts to the fiber area class in this trivial
region; its local representative can be kept fixed when patching
by exact forms, as in the proof of~\cite[Theorem~3.1]{Gompf}.
Thus $T$ is Lagrangian, with its product framing. Since $p, q$ are $1$--dimensional classes, their pullbacks to any $2$--sphere vanish, and $f^*[\sigma]$ also vanishes on spheres since $\Sigma_h$ is aspherical. Hence the form $\omega$ we constructed is symplectically aspherical.

Put $E=X\setminus\Int \, \nu T$. The vanishing cycle relations above
still hold in $E$.  Using the meridian calculation in Kirk's product model~\cite[Theorem~9.2]{Kirk},
with the base and fiber factors in that order, we can choose a meridian $\mu=[c_1^{-1},b_2]$.
The product torus $b_1\x c_1$ misses $T$, so $[b_1,c_1]=1$ in
$\pi_1(E)$. Together with $b_2=b_1^{-1}$, this gives $\mu=1$.
Regluing $\nu T \cong T^2\x D^2$ to $E$ by the original gluing map kills only the normal closure of
$\mu$, so $\pi_1(E) \cong \pi_1(X)$.
The framed $a_2'$ longitude $\lambda$ therefore represents $a_2$
in $\pi_1(E)\cong\Z^2\x\pi_1(\Sigma_h)$.
Luttinger surgery with slope $\mu+\lambda$ kills this generator and
gives a symplectic $X_T$ with
$\pi_1(X_T)\cong\Z\x\pi_1(\Sigma_h)$.

We claim that $X_T$ admits no symplectically aspherical form, by following arguments as in \cite{KRT}.  A classifying space $K$ for $\pi_1(X_T)$ can be obtained from $X_T$
by attaching $3$--cells to kill $\pi_2$, followed by higher dimensional cell attachments to kill all $\pi_k$, for $k \geq 3$. Let $\alpha\in H^2(X_T;\mathbb R)$ vanish on spherical classes.
A cellular $2$--cocycle representing $\alpha$ remains a cocycle after
attaching a $3$--cell precisely when its evaluation on the attaching
sphere is zero. This holds by our assumption, so $\alpha$ here extends over
all the $3$--cells. The higher cell attachments do not affect the
second cohomology, so we obtain an extension
$\widetilde\alpha\in H^2(K;\mathbb R)$ of $\alpha$.
The Eilenberg--Mac Lane space $K$ is homotopy equivalent
to $S^1\x\Sigma_h$, which is $3$--dimensional,
so $H^4(K;\mathbb R)=0$.\footnote{The argument breaks down for $X$ at this point: the classifying space $T^2\x\Sigma_h$ for $\pi_1(X)$ has nontrivial fourth cohomology.}
Thus $\widetilde\alpha^2=0$,
and restricting this equality to $X_T$ gives $\alpha^2=0$,
since restriction commutes with cup products.
Since a symplectic class on a closed $4$--manifold has nonzero square, this proves our claim that $X_T$ admits no symplectically aspherical form.
Varying $h$ in this example, we get an infinite family of examples, as $b_1(X)=2h+2$ and $b_1(X_T)=2h+1$. 
\end{proof}

\begin{rmk}[Further examples]\label{rmk:nonpreservation}
In fact, a careful look at the literature reveals that counterexamples to the conjecture that Luttinger surgery preserves symplectic asphericity are already implicit in several
constructions of exotic $4$--manifolds. For example, $X_1'$ in~\cite[Theorem~24]{BaykurKorkmaz}
is the extension over the torus of a genus-$2$ fibration with seven
critical points. The analogous cup product class has values
$(2,1)$, $(2,1)$ and $(1,2)$ on the reducible fiber components,
and $b_1=b_2^{-1}$ still holds, so the same form and meridian
arguments apply. The first of the four Luttinger surgeries gives
fundamental group $\Z^3$, already destroying symplectic asphericity on the way to an exotic copy of $\CP \# 4\CPb$.
More generally, any sequence of Luttinger surgeries from a
symplectically aspherical manifold to one with finite fundamental
group must lose symplectic asphericity at some step. Further
examples can be found in~\cite{BHArbitrary,BH}.
\end{rmk}

\section{Final remarks}\label{sec:addenda}

We conclude with a few variations that follow from the same local
models and boundary calculations.

\begin{rmk}\label{rmk:coefficients}
Keep the tori, framings, and chosen longitudes
in~\eqref{eq:fpslink}, respectively in~\eqref{eq:aplink}.
For any integers $k_i$, torus surgeries with slopes
$\mu_i+k_i\lambda_i$ give aspherical manifolds. The separating
$S^1\x\Sigma_2$ remains $\pi_1$-injective, and
\[
H_1=\oplus_{i=1}^{\ell}\Z_{|k_i|},\qquad
\ell=8\text{ for FPS},\quad \ell=6\text{ for AP}.
\]
Indeed, Lemma~\ref{lem:product} and the AP pants argument allow
arbitrary surgery coefficients, with the same boundary injections
and intersections. In both exteriors the meridians vanish in first
homology and the chosen longitudes form a basis, so the surgeries impose
precisely $k_i\lambda_i=0$. In particular, all choices $k_i=\pm1$ give
aspherical homology $S^2\x S^2$'s.

Applying Lemma~\ref{lem:multiplicity} in the first AP quadrant,
and changing the coefficient on the fifth torus, also gives
\[
M_m^p\text{ aspherical},\qquad H_1(M_m^p;\Z)=\Z_p
\quad(m\ge1,\ p\ge0).
\]
These are the manifolds in~\cite{AP}; the varying multiplicity is
confined to the first surgery. The second torus, which lies in the
same quadrant, is still surgered with multiplicity $1$.
\end{rmk}

\begin{rmk}\label{rmk:highergenus}
The FPS construction extends to $\Sigma_2\x\Sigma_h$, $h\ge2$
(and more generally the same local arguments apply to
$\Sigma_g\x\Sigma_h$, $g,h\ge2$), in a straightforward fashion.
Keep the eight surgeries on the first two handles of the second
factor, and for each $j=3,\ldots,h$ add the two pairs
\[
(a_1'\x c_j',\,c_j'),\qquad
(a_1''\x d_j',\,d_j'),
\]
with slopes $\mu+\lambda$. Denote the resulting manifold by $X_h$.
Cut $\Sigma_h$ into punctured tori $Q_j$ and a planar surface $R$.
Over each $Q_j$, Lemma~\ref{lem:product} and
Remark~\ref{rmk:otherfaces} assemble the two quadrants into an aspherical piece with
$\pi_1$-injective boundary $\Sigma_2\x\partial Q_j$.
Gluing these pieces to $\Sigma_2\x R$ proves asphericity. Taking $h=k+2$, $k\ge0$, gives aspherical fake
$\#_{2k+1}(S^2\x S^2)$'s.
\end{rmk}

We next show that asphericity is not specific to the
surgery slopes used in~\cite{FPS}: on the same tori, varying
both the coefficients and the longitude directions still gives
aspherical manifolds.

\begin{addendum}\label{rmk:longitudes}
On the fixed FPS tori, all slopes
$\mu_i+r_i\alpha_i+s_i\beta_i$, where $r_i,s_i\in\Z$ and
$\alpha_i,\beta_i$ are the product longitudes, give aspherical
manifolds. The separating $S^1\x\Sigma_2$ remains
$\pi_1$-injective.
\end{addendum}

\begin{proof}
Cut the base of a quadrant into an annulus $S^1\x I$ and a band,
and write $t$ for the loop around the circle factor.
The surgeries are regluings of $P\x S^1$ at successive levels
along $I$. Composing them in the order of these levels gives
$\Phi\in\operatorname{Aut}(U)$, where
$U=F\x\langle t\rangle$, $F=\pi_1(P)$ and $u=[\partial P]$.
Choose a basepoint on $\partial P\x S^1$. The regluings are
the identity near this boundary, so $\Phi$ fixes the based loops
$u,t$. Restoring the band gives asphericity by Whitehead's theorem
and the presentation
\begin{equation}\label{eq:mixedHNN}
V=\langle U,z\mid zfz^{-1}=\Phi(f),\ f\in F\rangle.
\end{equation}
With the based paths in the proof of Lemma~\ref{lem:multiplicity},
the two faces have images
\[
K=\langle F,\delta\rangle,\qquad L=\langle u,t,z\rangle,
\qquad\delta=z^{-1}tzt^{-1}.
\]
Since $\Phi$ fixes $\langle u,t\rangle$ pointwise, both $F$ and
$\Phi(F)$ meet this subgroup in $\langle u\rangle$, and neither
contains a nonzero power of $t$.

The relations in~\eqref{eq:mixedHNN} show that $\delta$ centralizes
$F$ and $u$ commutes with $t,z$.
For $n\ne0$, the word $\delta^n$ has no cancellable pair of
$z$ and $z^{-1}$ factors: the intervening nonzero powers of $t$ lie
in neither $F$ nor $\Phi(F)$. Thus the normal-form theorem
for~\eqref{eq:mixedHNN}~\cite[Theorem~11 and Corollaries~1--2]{Serre}
gives $K\cong F\x\Z$.
For words in $L$, such a cancellation requires the intervening
element of $\langle u,t\rangle$ to lie in $\langle u\rangle$.
These are exactly the cancellations allowed in
$\langle u\rangle\x\F(t,z)$, so this group also injects, with
image $L$.

Although $\Phi$ need not preserve $F$, it induces an automorphism
$\bar\Phi$ of $U/\langle t\rangle\cong F$.
Quotienting~\eqref{eq:mixedHNN} by $t=1$ gives
$F\rtimes_{\bar\Phi}\langle z\rangle$, injectively on $F$.
The quotient kills $\delta$ and sends $L$ to $\langle u,z\rangle$,
which meets $F$ in $\langle u\rangle$.
Thus an element $f\delta^n$ of $K$ which also lies in $L$ must
have $f\in\langle u\rangle$. Conversely, $u,\delta$ both lie in
$L$, so $K\cap L=\langle u,\delta\rangle$.
The FPS gluings now apply as before.
\end{proof}

The same boundary injections also give the following geometric consequence.

\begin{addendum}\label{add:curvature}
Every manifold in Theorem~\ref{thm:families} or Addendum~\ref{rmk:longitudes} has a fundamental group containing a
subgroup isomorphic to $\Z\x\pi_1(\Sigma_2)$.
In particular, these groups are infinite, and none of these manifolds
admits a Riemannian metric of strictly negative sectional curvature.
\end{addendum}

\begin{proof}
In all these constructions, the separating $S^1\x\Sigma_2$ remains
$\pi_1$-injective, giving the subgroup in the statement and hence the conclusion that $\pi_1$ is infinite,
without using asphericity. This subgroup contains $\Z^2$.
However, by Preissmann's theorem~\cite[Theorem~9]{Preissmann}, every nontrivial
abelian subgroup of the fundamental group of a closed manifold with
strictly negative sectional curvature is infinite cyclic.
\end{proof}

\begin{rmk}[Failure of boundary injectivity]\label{rmk:pinwheels}
The pinwheel constructions in~\cite{FS} give a neat way to present
various constructions of fake and exotic copies of rational surfaces,
and make it visible where our asphericity arguments fail for truly
exotic copies---those with trivial $\pi_1$, and therefore homeomorphic
to the rational surface.
For instance, in the blown-up block used
for the $\CP\#2\CPb$ construction,
an additional $2$-handle is attached along the boundary circle of
one punctured torus, adding the relation $[x,y]=1$;
see~\cite[Sections~6 and~7.2]{FS}.
This commutator is nontrivial in the corresponding product face
$\Sigma_1^1\x S^1$, but is trivial in the block, and remains
so after the subsequent torus surgeries. Thus this face is not
$\pi_1$-injective, and Whitehead's theorem does not apply to this
decomposition.
Similarly, one can spot this failure using~\cite[Theorem~10.3 and
Lemma~10.6(d)]{Kirk}: in the notation there, the nontrivial boundary
commutator $[b_1,a_2]$ maps to $[\tau,y]=1$ in the genus $2$
surface exterior, so this boundary inclusion is not $\pi_1$-injective.

Similarly, the subgroup obstruction of Addendum~\ref{add:curvature}
requires the separating boundary to be $\pi_1$-injective in
both adjoining pieces. This hypothesis fails, e.g., for the exotic
signature zero $4$--manifolds constructed in~\cite{BH}, which admit
horizontal and vertical decompositions along $S^1\x\Sigma_2$
and $S^1\x\Sigma_g$, respectively; see~\cite[Figure~13]{BH}.
In each case, the inclusion into the piece containing the Lefschetz
kernel is not $\pi_1$-injective: the surgery relations and Lefschetz
thimbles kill some nontrivial boundary loops. The subgroup obstruction
does not work for such decompositions.
\end{rmk}

\begin{rmk}[Other examples] \label{rmk:further}
There are other constructions of fake $S^2\x S^2$'s, obtained through torus surgeries in similar models, which we have not studied here. These include those in~\cite{Akhmedov,LP}, where no claim of exotica is made, and~\cite{Wuebben}, which indeed claims simple connectivity and exoticness for a variant of the construction in~\cite{LP}. I imagine that arguments similar to those we employed in this paper can also be applied to many of these examples. A few people have also pursued constructions using surgeries along Lagrangian Klein bottles~\cite{Nemirovski} in surface products or surface bundle models, though I know of no published candidates obtained this way. I expect that the arguments of Section~\ref{sec:gluing} can be adapted to such constructions as well.
\end{rmk}

\bigskip
\bibliographystyle{amsplain}
\bibliography{refs}

\end{document}